\documentclass[english]{amsart}
\usepackage[T1]{fontenc}
\usepackage[utf8]{inputenc}
\usepackage{xcolor}
\usepackage{babel}
\usepackage{amsbsy}
\usepackage{amstext}
\usepackage{amsthm}
\usepackage{amssymb}
\usepackage[pdfusetitle,
 bookmarks=true,bookmarksnumbered=false,bookmarksopen=false,
 breaklinks=false,pdfborder={0 0 0},pdfborderstyle={},backref=false,colorlinks=true]
 {hyperref}
\hypersetup{
 linkcolor=brown, citecolor=blue, pdfstartview={FitH}, hyperfootnotes=false, unicode=true}

\makeatletter
\numberwithin{equation}{section}
\numberwithin{figure}{section}
\theoremstyle{plain}
\newtheorem{thm}{\protect\theoremname}[section]
\theoremstyle{definition}
\newtheorem{defn}[thm]{\protect\definitionname}
\theoremstyle{definition}
\newtheorem{example}[thm]{\protect\examplename}
\theoremstyle{plain}
\newtheorem{lem}[thm]{\protect\lemmaname}
\theoremstyle{plain}
\newtheorem{prop}[thm]{\protect\propositionname}
\theoremstyle{plain}
\newtheorem{cor}[thm]{\protect\corollaryname}
\theoremstyle{definition}
\newtheorem{problem}[thm]{\protect\problemname}

\usepackage{tikz-cd}
\usepackage{bbm}
\usepackage{tikz}
\usetikzlibrary{positioning, fit}
\newtheorem{theoremalpha}{Theorem}
\newtheorem{problemalpha}[theoremalpha]{Problem}

\newtheorem{propalpha}[theoremalpha]{Proposition}

\makeatother

\providecommand{\corollaryname}{Corollary}
\providecommand{\definitionname}{Definition}
\providecommand{\examplename}{Example}
\providecommand{\lemmaname}{Lemma}
\providecommand{\problemname}{Problem}
\providecommand{\propositionname}{Proposition}
\providecommand{\theoremname}{Theorem}

\begin{document}
\title[Residual properties in the Weihrauch lattice]{Residual properties of finitely generated groups in the Weihrauch
lattice}
\author{Emmanuel Rauzy}
\address{Universit\'e Paris-Est Cr\'eteil\\ LACL}
\email{emmanuel.rauzy@u-pec.fr}
\thanks{The present work stems from a question raised by Vincent Guirardel
during the defense of my PhD thesis. It took me several years to understand
the question and its relevance. Thanks are also due to Vasco Brattka,
whose explanations have greatly improved my understanding of Weihrauch
reducibility.\\
Most of this work was conducted while the author was supported by 
an Alexander von Humboldt fellowship.}
\begin{abstract}
Consider, on the space of marked groups, the map $\mathrm{Res}_{\mathcal{C}}$
which associates to a marked group its greatest residually-$\mathcal{C}$
quotient, for different sets $\mathcal{C}$ of groups. Except for
trivial cases, this map is discontinuous. We use the Weihrauch lattice
to quantify how discontinuous it is. We show that equational noetherianity
of $\mathcal{C}$ and whether the set of residually-$\mathcal{C}$
groups is a quasivariety both can be characterized in terms of the
position of $\mathrm{Res}_{\mathcal{C}}$ within the Weihrauch lattice.
We give exact classifications of $\mathrm{Res}_{\mathcal{C}}$, for
$\mathcal{C}$ one of: the set of finite groups, of nilpotent groups,
of $k$-nilpotent groups, $k\ge1$, of finitely presentable groups,
of LEF groups, of torsion free groups. 
\end{abstract}

\maketitle

\section{Introduction}

Let $\mathcal{C}$ be a class of groups. 

A group $G$ is residually-$\mathcal{C}$ if for every non-identity
element $g$ of $G$ there is a quotient of $G$ in $\mathcal{C}$
in which $g$ has a non-identity image. 

Every group $G$ has a greatest residually-$\mathcal{C}$ quotient,
obtained by taking the quotient of $G$ by the intersection of all
kernels of onto homeomorphisms from $G$ to groups of $\mathcal{C}$,
i.e., by the following normal subgroup:
\[
\bigcap_{f\colon G\twoheadrightarrow H,\,H\in\mathcal{C}}\ker(f).
\]
We denote by $\mathrm{Res}_{\mathcal{C}}(G)$ the group defined this
way, and by $\mathrm{Res}_{\mathcal{C}}$ the map sending $G$ to
this quotient. 

\vspace{0.2cm}

Note that, even though proving the existence of the group $\mathrm{Res}_{\mathcal{C}}(G)$
is trivial, the definition given above provides us with very little
information. Indeed, ``the set of all kernels of onto homomorphisms
from $G$ to groups in $\mathcal{C}$'' always exists, but it may
be very hard to characterize or understand this set. 

In some sense, the existence of the group $\mathrm{Res}_{\mathcal{C}}(G)$
is non-constructive. (This statement could be made precise in various
ways. For instance \cite{Bauer2000} provides a constructive framework
which coincides with the Weihrauch analysis of problems.)

In the present paper, we provide a topological analysis of the map
$\mathrm{Res}_{\mathcal{C}}$ that gives a way to quantify \emph{how
non-constructive} is the existence of $\mathrm{Res}_{\mathcal{C}}(G)$. 

One of our motivations comes from the study of computability. We will
not insist on this aspect, but we quote here a result of Slobodskoi. 
\begin{thm}
[Slobodskoi, \cite{Slobodskoi1981}]There exists a finitely presented
group whose greatest residually finite quotient is not recursively
presented. 
\end{thm}

This result, and other related results from \cite{Rauzy_2021,Bartholdi2025},
all revolve around leveraging the discontinuity of the map $\mathrm{Res}_{\mathcal{C}}$
to obtain computability theoretical results. Our hope is that the
present study in turn provides a dichotomy between classes $\mathcal{C}$
which give rise to ``a sufficiently discontinuous $\mathrm{Res}_{\mathcal{C}}$'',
so that certain computability theoretical phenomena occur, and the
other ones, for which the map $\mathrm{Res}_{\mathcal{C}}$ is too
tame to permit these computability theoretical phenomena.

\vspace{0.2cm}

In order to have a nice topological framework to work with, we will
study the map $\mathrm{Res}_{\mathcal{C}}$ on the set of marked groups.
A $k$-marked group is a pair $(G,S)$, where $G$ is a finitely generated
group and $S\in G^{k}$ is a generating tuple of $G$. A morphism
of $k$-marked groups $(G,(s_{1},...,s_{k}))\rightarrow(H,(s'_{1},...,s'_{k}))$
is a group morphism $f\colon G\rightarrow H$ which satisfies $f(s_{i})=s'_{i}$
for all $1\le i\le k$.

The marked quotient relation 
\[
(G,S)\succeq(H,S')\iff\exists f\colon(G,S)\rightarrow(H,S')
\]
is easily seen to be an order relation. The set of $k$-marked groups
forms a poset which is in fact a complete lattice isomorphic to the
lattice of normal subgroups of a rank-$k$ free group. In this context,
we have the equality
\[
\mathrm{Res}_{\mathcal{C}}(G,S)=\sup\left\{ (H,S')\in\mathcal{C}\mid(G,S)\succeq(H,S')\right\} .
\]

The set of marked groups is naturally equipped with a Polish topology,
the resulting topological space is called the \emph{space of marked
groups}. 

The main problem of this paper is the following one:

\begin{problemalpha}	

What properties of $\mathcal{C}$ can be read off of the topological
properties of $\mathrm{Res}_{\mathcal{C}}$? 

\end{problemalpha}	

We could ask for instance: when is $\mathrm{Res}_{\mathcal{C}}$ continuous,
when are preimages of open sets $\Sigma_{2}^{0}$-sets, and so on.

In order to obtain finer classifications than these, we use the notion
of continuous Weihrauch reduction, which can be thought of as a generalization
of Wadge reduction adapted to functions. We discuss Weihrauch reduction
in more details in Section \ref{subsec:Weihrauch-reduction} of this
introduction. 

\subsection{The logical point of view on the map $\mathrm{Res}_{\mathcal{C}}$}

Another point of view on the map $\mathrm{Res}_{\mathcal{C}}$ follows
the lines of \cite{Cha}. 

A relation in a marked group $(G,S)$ is an element of the group $\mathbb{F}_{S}$,
the free group over $S$, which defines the identity in $G$. A subset
of $\mathbb{F}_{S}$ is the set of relations of a marked group if
and only if it is a normal subgroup of $\mathbb{F}_{S}$. 

Thus we can say that a set $R\subseteq\mathbb{F}_{S}$ of relations
\emph{implies} another relation $w\in\mathbb{F}_{S}$ when $w$ belongs
to the normal subgroup generated by $R$. 

When focusing on groups in a certain class $\mathcal{C}$, it is natural
to consider implication modulo $\mathcal{C}$: 
\begin{defn}
Let $R$ be a subset of a free group $\mathbb{F}_{S}$ and $w$ an
element of $\mathbb{F}_{S}$. We say that the relations of $R$ \emph{imply
$w$ relative to} $\mathcal{C}$ if every marked group in $\mathcal{C}$
which satisfies the relations of $R$ also satisfies $w=1$.
\end{defn}

The marked group $\mathrm{Res}_{\mathcal{C}}(G,S)$ is the marked
sgroup whose relations are exactly those that are consequences relative
to $\mathcal{C}$ of the relations of $(G,S)$. 

Thus the study of the map $\mathrm{Res}_{\mathcal{C}}$ can be understood
as the problem of understanding how hard it is, starting with a set
of relations which is sound with respect to the usual implication
between relations, to recover a set of relations which is sound with
respect to the implication rule induced by $\mathcal{C}$.

\subsection{\label{subsec:Weihrauch-reduction}Weihrauch reduction }

An informal definition of continuous Weihrauch reducibility is the
following: 
\[
f\le_{W}g
\]
if and only if $f$ can be computed by ``doing something continuous,
applying $g$, and then doing something continuous again''. What
makes this informal statement non-trivial to formalize is the fact
that what we mean by ``something continuous'' is in fact a \emph{continuous
multi-function},\emph{ }instead of being simply a function, and that
the notion of continuous multi-function used to define Weihrauch reduction
is set in the category of \emph{represented spaces} instead of the
category of topological spaces. 

We will explain this in more details below. 

Let us first give some examples to motivate the introduction of Weihrauch
reduction. 

For $X$ and $Y$ Polish spaces, let us call a function $f\colon X\rightarrow Y$
$\Sigma_{n}^{0}$-measurable if the preimage of an open set by $f$
is $\Sigma_{n}^{0}$. 

Also, let $\mathrm{Lim}$ denote the limit map on Baire space:
\[
\mathrm{Lim}:\subseteq\mathbb{N}^{\mathbb{N}}\rightarrow\mathbb{N}^{\mathbb{N}},\,\langle p_{0},p_{1},p_{2},...\rangle\underset{n\rightarrow\infty}{\longrightarrow}\lim p_{n},
\]
where $\langle\rangle$ denotes a bijection between $\mathbb{N}^{\mathbb{N}}$
and $(\mathbb{N}^{\mathbb{N}})^{\mathbb{N}}$. The domain of $\mathrm{Lim}$
is the set sequences of $\mathbb{N}^{\mathbb{N}}$ which encode via
$\langle\rangle$ a converging sequence. Also, we denote by $\mathrm{Lim}^{(n)}$
the $n$-fold composition of $\mathrm{Lim}$.
\begin{thm}
[\cite{Brattka2004}]A function $f\colon X\rightarrow Y$ between
Polish spaces is $\Sigma_{n}^{0}$-measurable if and only if 
\[
f\le_{W}\mathrm{Lim}^{(n)}.
\]
\end{thm}

In other words, $\mathrm{Lim}^{(n)}$ is $\Sigma_{n}^{0}$\emph{-measurable-complete}.
Quoting from \cite{Brattka2021a}, ``Weihrauch reducibility offers
a refinement of the Borel hierarchy very much in the same way as many-one
reducibility yields a refinement of the Kleene hierarchy''. 

The Weihrauch lattice has a rich structure, with algebraic operations
on problems, closure operators that capture other reducibility notions
\cite{Brattka2021a}. Many types of behaviors of discontinuous functions
can be described thanks to a classification in the Weihrauch lattice.

$ $

Let us come back to a formal definition of Weihrauch reduction. 

A \emph{represented space }is a set $X$ equipped with a partial surjection
from Baire space $\delta:\subseteq\mathbb{N}^{\mathbb{N}}\rightarrow X$.
When $\delta(p)=x$, we say that $p$ is a $\delta$\emph{-name} of
$x$. Let $f\colon X\rightarrow Y$ be a function between represented
spaces $(X,\delta)$ and $(Y,\tau)$. A \emph{realizer} of $f$ is
a partial function $F:\subseteq\mathbb{N}^{\mathbb{N}}\rightarrow\mathbb{N}^{\mathbb{N}}$
such that 
\[
\forall p\in\mathrm{dom}(\delta),\,f(\delta(p))=\tau(F(p)).
\]
In words: $F$ maps the name of a point in $X$ to a name of its image. 

A multi-function $f:\subseteq X\rightrightarrows Y$ is simply a relation
$R\subseteq X\times Y$ which we treat as a multi-valued function
via the formula $f(x)=\{y\in Y\mid R(x,y)\}$. A \emph{realizer} of
a multi-function $f$ is a partial function $F:\subseteq\mathbb{N}^{\mathbb{N}}\rightarrow\mathbb{N}^{\mathbb{N}}$
such that 
\[
\forall p\in\mathrm{dom}(\delta)\cap\delta^{-1}(\mathrm{dom}(f)),\,\tau(F(p))\in f(\delta(p)).
\]
In words: $F$ maps the name of a point in $X$ to a name of one of
its images. 

A (multi-)function is called \emph{continuously realizable }if it
has a continuous realizer. 
\begin{defn}
[Continuous Weihrauch reducibility]Let $f:\subseteq X\rightrightarrows Y$
and $g:\subseteq Z\rightrightarrows W$ be multi-functions between
represented spaces. We say that $f$ \emph{continuously Weihrauch
reduces to} $g$, and write 
\[
f\le_{W}g,
\]
if there is a represented space $V$ and two continuously realizable
multi-functions $h:\subseteq V\times W\rightrightarrows Y$ and $k:\subseteq X\rightrightarrows V\times Z$
such that 
\[
\forall x\in X,\,\forall(v,z)\in k(x),\,f(x)\in h(v,g(z)).
\]
\end{defn}

The role of the space $V$ and of the element $v$ in the above is
to ``pass the input of $f$ through $g$''. 

If $f\le_{W}g$ and $g\le_{W}f$, we say that $f$ and $g$ are \emph{Weihrauch
equivalent}, and write $f\equiv_{W}g$. 

\subsection{\label{subsec:Admissible-representations-of}Admissible representations
of topological spaces}

The definition of Weihrauch reduction given above is set in the category
of represented spaces. We can express Weihrauch reduction for functions
between topological spaces by establishing a correspondence between
some topological spaces (here it will be sufficient to deal with $\mathrm{T}_{0}$
second countable spaces) and some representations (the \emph{admissible
representations}). 

First, to any representation $\rho$ of a set $X$ we can associate
a topology, the \emph{final topology of the representation}, given
by
\[
U\subseteq X\text{ is open}\ensuremath{\iff\rho^{-1}(U)}\text{ is open in Baire space}.
\]

Secondly, we associate representations to topological spaces.

Let $\rho$ and $\tau$ be two representations of a set $X$. We say
that $\rho$ \emph{(continuously) translates} to $\tau$ if the function
\[
\mathrm{id}_{X}:(X,\rho)\rightarrow(X,\tau)
\]
has a continuous realizer. Two representations of a set $X$ are called
equivalent if each one translates to the other one. 
\begin{defn}
[\cite{Kreitz1985,Schroeder2002}]Let $X$ be a topological space.
A representation $\rho:\subseteq\mathbb{N}^{\mathbb{N}}\rightarrow X$
of $X$ is \emph{admissible} if and only if
\begin{itemize}
\item It is continuous, 
\item Any continuous representation of $X$ translates to $\tau$. 
\end{itemize}
\end{defn}

Any second countable $\mathrm{T}_{0}$ space can be equipped with
a \emph{standard representation} (Definition \ref{def:Kreitz-Weihrauch, Standard Rep }),
which is always admissible, and admissible representations are unique
up to equivalence. And we have: 
\begin{thm}
[\cite{Kreitz1985}]Let $X$ and $Y$ be second countable $\mathrm{T}_{0}$
spaces. Then a function $f\colon X\rightarrow Y$ is continuous if
and only if it is continuously realizable with respect to any admissible
representations of $X$ and $Y$. 
\end{thm}

In other words, we have an equivalence of categories between second
countable $\mathrm{T}_{0}$ spaces with continuous functions as morphisms
on the one hand, and represented spaces equipped with equivalence
classes of standard representations on the other hand. 

Moving to the category of represented spaces has two main advantages:
computability theory can easily be developed (and this was the original
motivation of Kreitz and Weihrauch \cite{Kreitz1985}), and continuous
multi-functions can be very naturally defined, in a way that seems
hard to do in the category of topological spaces (see \cite{Brattka1994}
for more details on this).

By results of Schr{\"o}der \cite{Schroeder2002,Schroeder2021}, admissible
representations can be constructed exactly for those $\mathrm{T}_{0}$
topological spaces whose sequentialization is the quotient of a second
countable space. In the present article, we will not need this generality.

\subsection{Representations of marked groups }

The back and forth between represented spaces and topological spaces
can be though of in the following way: each type of description of
the elements of a set $X$ yields a topology on $X$ (the sets of
open sets being the set of ``recognizable properties''), and conversely
each (second countable $\mathrm{T}_{0}$) topology yields a way to
describe the elements of $X$ (an element is described by the basic
open sets to which it belongs).

In this setting, the topology of the space of marked groups arises
in a very natural way: it is the topology associated to marked groups
defined by their word problem. In Section \ref{sec:Different-representations-of marked groups},
we define a representation of the set of marked group by saying that
``the name of a marked group is a sequence which encodes the characteristic
function of its set of relations'', and the final topology of this
representation is precisely the topology of the space of marked groups. 

Another topology that is used several times in the present paper is
the Scott topology on the set of marked groups, which admits as a
base the set of basic sets of the form ``those marked groups which
satisfy the relations $\{r_{1},...,r_{n}\}$'', for any finite set
of relations $\{r_{1},...,r_{n}\}$. This topology is the final topology
of the representation associated to \emph{group presentations}: a
marked group is described by a (countably infinite) presentation. 

\subsection{Results }

We denote by $\mathcal{RC}$ the set of residually-$\mathcal{C}$
marked groups. 

\begin{propalpha}	

As soon as $\mathcal{RC}$ is neither all marked groups nor the singleton
containing the trivial group, the map $\mathrm{Res}_{\mathcal{C}}$
is discontinuous. 

\end{propalpha}\begin{theoremalpha}	

The map $\mathrm{Res}_{\mathcal{C}}$ is $\Sigma_{2}^{0}$-measurable
if and only if $\mathcal{RC}$ is a quasivariety. 

\end{theoremalpha}

\begin{theoremalpha}	

For the following quasivarieties, $\mathrm{Res}_{\mathcal{C}}$ is
$\Sigma_{2}^{0}$-complete: 
\begin{itemize}
\item the set of finitely generated $k$-solvable groups, for $k>2$, 
\item the set of finitely generated exponent $k$ groups, for $k$ a big
enough odd number, 
\item the set of finitely generated LEF groups, 
\item the set of finitely generated torsion free groups, 
\item the set of finitely generated left orderable groups. 
\end{itemize}
\end{theoremalpha}

Let $\mathrm{Lim}_{\mathbb{N}}:\subseteq\mathbb{N}^{\mathbb{N}}\rightarrow\mathbb{N}$
be the function that maps a converging sequence of natural numbers
to its limit. 

\begin{theoremalpha}	

The Weihrauch reduction 
\[
\mathrm{Res}_{\mathcal{C}}\ge_{W}\mathrm{Lim}_{\mathbb{N}}
\]
 holds if and only if $\mathcal{C}$ is not equationally Noetherian. 

\end{theoremalpha}

For each countable ordinal $\alpha$, the function $\mathrm{Lim}_{\searrow\alpha}$
maps a weakly descending sequence in $\alpha$ to its limit. See Section
\ref{sec:Discretization-problem-associate Alexandrov Top} for a precise
definition. 

\begin{theoremalpha}	

We get the following Weihrauch reductions: 
\begin{itemize}
\item For $\mathcal{C}$ the set of metabelian groups: 
\[
\mathrm{Res}_{\mathcal{C}}\ge_{W}\mathrm{Lim}_{\searrow\omega^{\omega}}.
\]
\item For $\mathcal{C}$ the set of $p$-nilpotent groups, for any $p\ge1$,
\[
\mathrm{Res}_{\mathcal{C}}\equiv_{W}\mathrm{Lim}_{\searrow\omega^{2}}.
\]
\item For $\mathcal{C}$ a finite set of finite groups, 
\[
\mathrm{Res}_{\mathcal{C}}\equiv_{W}\mathrm{Lim}_{\searrow\omega}.
\]
\end{itemize}
\end{theoremalpha}

The ordinals that appear in the above proof are related to the ordinal
rank of the lattice of $\mathcal{C}$ marked groups, when $\mathcal{C}$
is an equationally Noetherian class. For a similar context in which
these ordinals were found relevant, see \cite{ho2026}. 

\begin{theoremalpha}	

For $\mathcal{C}$ any of the sets of finite groups, of nilpotent
groups or of finitely presented groups, $\mathrm{Res}_{\mathcal{C}}$
is $\Sigma_{3}^{0}$-measurable complete. 

\end{theoremalpha}

We also obtain some results in classical descriptive set theory. The
following answers a problem of \cite{Benli2019}. 

\begin{theoremalpha}	

The set of residually finite groups is $\Sigma_{3}^{0}$-complete.
The same goes for the sets of residually nilpotent and residually
finitely presented groups. 

\end{theoremalpha}

We summarize our results in Figure \ref{fig:Summary}. 

\begin{figure}
 \begin{center} 	\scalebox{1}{ 		\begin{tikzpicture}[node distance=1cm, >=stealth, font=\small]
			\node (dots) {$\vdots$}; 			\node (exists0) [below right=0.05cm and 1.2cm of dots] {$\exists\mathrm{Res}_{\mathcal{C}}$}; 			\node (limprime) [below=of dots] {$\mathrm{Lim}'$}; 			\node (exists1) [below right=0.5cm and 1cm of limprime] {$\exists?\mathrm{Res}_{\mathcal{C}}$}; 			\node (lim) [below=2cm of limprime] {$\mathrm{Lim}$}; 			\node (exists2) [below right=0.5cm and 1cm of lim] {$\exists?\mathrm{Res}_{\mathcal{C}}$}; 			\node (limN) [below= 1.6cm of lim] {$\mathrm{Lim}_{\mathbb{N}}$}; 			\node (limoo) [below=1.2cm of limN] {$\mathrm{Lim}_{\searrow\omega^{\omega}}$}; 			\node (limo2) [below=of limoo] {$\mathrm{Lim}_{\searrow\omega^2}$}; 			\node (limo) [below=of limo2] {$\mathrm{Lim}_{\searrow\omega}$}; 			\node (secret)[right =9cm of lim]  {}; 			\node (secret2)[below =0.15cm of limN]  {}; 			\node (secret3)[right =9cm of secret2]  {};
			\node (res1) [below right =-0.48cm and -0.1cm of limprime] {$\equiv_W \mathrm{Res}_{\mathcal{C}}$, 			$\mathcal{C}\in\{ \text{Finite, Finitely presented, Nilpotent}\}$}; 			\node (res2) [below right =-0.49cm and -0.1cm of  lim]  {$\equiv_W \mathrm{Res}_{\mathcal{C}}$, 			$\mathcal{C}\in\{ \text{Torsion Free, LEF, $k$-solvable, ...}\}$}; 			\node (res3) [above right =0.2cm and 0.6cm of limoo]  {$\mathrm{Res}_{\mathrm{Metabelian}}$}; 			\node (res4) [below right =-0.52cm and -0.1cm of limo2] {$\equiv_W \mathrm{Res}_{\mathrm{Abelian}}\equiv_W \mathrm{Res}_{k-\mathrm{Nilpotent}}$}; 			\node (res5) [below right =-0.53cm and -0.1cm of  limo] {$\equiv_W \mathrm{Res}_{\mathcal{C}}, 			\mathcal{C} \text{ a finite set of finite groups}$};
			\foreach \i/\j in {dots/limprime, limprime/lim, lim/limN, limN/limoo, limoo/limo2, limo2/limo} 			\draw[->] (\i) -- (\j); 			\draw[->] (limN) -- (res3); 			\draw[->] (res3) -- (limoo);
			\node[draw, rounded corners, fit=(lim) (res2) (limo) (secret), inner sep=18pt, label=above right:{}] (box1) {}; 			\node[anchor=north east, font=\small, inner sep=8pt] at (box1.north east) {\textbf{Quasivarieties}}; 			\node[draw, rounded corners, fit=(limoo) (res3) (limo2) (res4) (limo) (res5) (secret3), inner sep=6pt, label= above:{}] (box2) {}; 			\node[anchor=north east, font=\small, inner sep=8pt] at (box2.north east) {\textbf{Equational Noetherianity}};
		\end{tikzpicture} 	} \end{center} \caption{\label{fig:Summary}Summary}

\end{figure}

\subsection*{Notations}

Let $\mathcal{C}$ be a class of groups. 

$\mathrm{Res}_{\mathcal{C}}$ denotes the function that maps a marked
group to its greatest residually $\mathcal{C}$ quotient. Without
further precision, we will consider that both the domain and co-domain
of $\mathrm{Res}_{\mathcal{C}}$ are equipped with the topology of
the space of marked groups. 

$\mathcal{RC}$ is the set of residually-$\mathcal{C}$ marked groups. 

$\le_{W}$ denotes continuous Weihrauch reduction. 

$\le$ denotes continuous translation between representations. 

$\mathcal{G}_{\mathrm{WP}}$ denotes the space of marked groups, which
is associated to the representation associated to groups given by
their word-problem. 

$\mathcal{G}_{\mathrm{Pres}}$ denotes the set of marked groups equipped
with the Scott topology, and with the corresponding representation
associated to presentations of groups.

By adding an exponent $k$ to either one of the above, $\mathcal{G}_{\mathrm{WP}}^{k}$
or $\mathcal{G}_{\mathrm{Pres}}^{k}$, we indicate that we restrict
our attention to $k$-marked groups.

\section{group theoretical background}

\subsection{\label{sec:Marked-groups}Marked groups}

Fix $k\in\mathbb{N}$. A $k$-\emph{marked group} is a finitely generated
group together with a $k$-tuple of elements that generate it. 

A \emph{morphism of $k$-marked groups} from $(G,(s_{1},...,s_{k}))$
to $(H,(s'_{1},...,s'_{k}))$ is a group morphism $\phi$ between
$G$ and $H$ which satisfies that $\phi(s_{i})=s_{i}'$ for $i=1..k$.
Such a morphism is an isomorphism if the underlying group morphism
is a group isomorphism, and marked groups are considered up to isomorphism. 

Note that there is at most one morphism from a $k$-marked group to
another $k$-marked group, and that if there are morphisms $(G,S)\rightarrow(H,S')$
and $(H,S')\rightarrow(G,S)$, then $(G,S)$ and $(H,S')$ are isomorphic
as marked groups: isomorphism classes of marked groups form a partially
ordered set for the quotient relation.

Let $(\mathcal{G}_{k},\rightarrow)$ be the poset of isomorphism classes
of $k$-marked groups. 

Remark that any isomorphism class of marked group can be identified
with a normal subgroup of a rank $k$ free group $\mathbb{F}_{k}$,
by associating to a marked group $(G,S)$ the kernel of the morphism
$\mathbb{F}_{k}\rightarrow G$ which maps a basis of $\mathbb{F}_{k}$
to $S$. It follows from this that the poset $(\mathcal{G}_{k},\rightarrow)$
is isomorphic to the lattice of normal subgroups of $\mathbb{F}_{k}$
under reverse inclusion. 

Let $(\mathcal{G},\rightarrow)$ be the poset of isomorphism classes
of marked groups, obtained by taking the disjoint union of the sets
$(\mathcal{G}_{k},\rightarrow)$. 

\subsection{Quasivarieties}

A \textbf{(group) quasi-identity }is a sentence of the form 
\begin{equation}
\bigwedge_{i=1}^{n}w_{i}=1\implies r=1,\label{eq:quasi-identity}
\end{equation}
for $(w_{1},...,w_{n},r)$ a finite tuple of elements of a free group
$\mathbb{F}_{S}$. A group $G$ satisfies the quasi-identity defined
by $(w_{1},...,w_{n},r)\in\mathbb{F}_{S}^{n+1}$ if for any substitution
of the generators of $\mathbb{F}_{S}$ by elements of $G$ the sentence
(\ref{eq:quasi-identity}) is satisfied. 

A class of groups $\mathcal{QV}$ is called a \textbf{group quasivariety}
if there exists a set $L$ of quasi-identities such that $\mathcal{QV}$
is the class of all groups that satisfy the quasi-identities of $L$. 

Equivalently \cite{Burris1981}, a group quasivariety is a class of
groups stable by forming subgroups, direct products, and ultraproducts. 

A \textbf{group variety} is a quasivariety that can be defined by
a set of quasi-identities of the form $1=1\implies r=1$. Such quasi-identities
are called \textbf{identities}. Equivalently, by Birkhoff's Variety
Theorem \cite{Burris1981}, a variety is a class of groups stable
by quotients, subgroups and Cartesian products. 
\begin{example}
The following are varieties:
\begin{itemize}
\item abelian groups,
\item $k$-nilpotent groups, 
\item $k$-solvable groups, 
\item exponent $k$-group. 
\end{itemize}
The following are quasivarieties: 
\begin{itemize}
\item torsion free groups, 
\item left and bi-orderable groups \cite{Howie1982}, 
\item locally indicable groups \cite{Howie1982}, 
\item LEF groups, 
\item groups that do not contain Thompson's Group F. 
\end{itemize}
\begin{example}
The set of residually free groups is a quasivariety, because free
groups are equationally noetherian (see Lemma \ref{lem:Eq noeth -> Res C is Qvar}).
The limit groups are known to be those residually free groups which
are \emph{commutative transitive}: for all $(a,b,c)$ with $b\ne1$,
if $a$ and $b$ commute and $b$ and $c$ commute, then $a$ and
$c$ commute. This expression is close to being a quasi-identity,
but not quite: 
\[
\forall a,b,c,\,([a,b]=1\wedge[b,c]=1)\implies(b=1\vee[a,c]=1).
\]
The set of limit groups is in fact not a quasivariety, as $\mathbb{F}_{2}\times\mathbb{F}_{2}$
is not a limit group. 
\end{example}

\end{example}

\subsection{Equational Noetherianity }

For $n$ in $\mathbb{N}$, consider a tuple $(X_{1},...,X_{n})$ of
unknowns, and let $\mathbb{F}_{n}$ be the free group on $(X_{1},...,X_{n})$. 

An \emph{equation} is an element $W\in\mathbb{F}_{n}$. 

If $(g_{1},...,g_{n})$ are elements in a group $G$, let $W(g_{1},...,g_{n})$
be the element of $G$ obtained by substituting in $W$ each element
$X_{i}$ by the corresponding $g_{i}$. We say that $(g_{1},...,g_{n})$
is a solution to $W$ if $W(g_{1},...,g_{n})=1$. 

A set $Sys\subseteq\mathbb{F}_{n}$ is called a \emph{system of equations}.
For a system $Sys$ of equations and a set of groups $\mathcal{C}$,
we define the \emph{algebraic set $V_{Sys}(\mathcal{C})$} associated
to $Sys$ as 
\[
V_{Sys}(\mathcal{C})=\{(g_{1},...,g_{n})\in G\mid G\in\mathcal{C},\,W(g_{1},...,g_{n})=1\,\forall W\in Sys\}.
\]

\begin{defn}
[\cite{Baumslag1999}] A group $G$ is \emph{equationally noetherian}
if for every $n$ and every tuple $(X_{1},...,X_{n})$ of unknown,
every system of equations with unknowns $(X_{1},...,X_{n})$ over
$G$ is equivalent to a finite subsystem: 
\[
\forall Sys\subseteq\mathbb{F}_{n},\exists Sys_{0}\subseteq Sys,\,Sys_{0}\text{ is finite }\&\,V_{Sys}(\{G\})=V_{Sys_{0}}(\{G\}).
\]
\end{defn}

The above definition was extended to families in \cite{Groves2019}. 
\begin{defn}
[\cite{Groves2019}, Definition A]A set $\mathcal{C}$ of groups
is \emph{equationally noetherian} if for every $n$ and every tuple
$(X_{1},...,X_{n})$ of unknown, every system of equations with unknowns
$(X_{1},...,X_{n})$ over $\mathcal{C}$ is equivalent to a finite
subsystem. 
\end{defn}

We will mostly use the following characterization of equational noetherianity,
which appears in \cite{Houcine2007} for equationally noetherian groups,
and in \cite{Groves2019} for equationally noetherian families: 
\begin{lem}
[\cite{Houcine2007,Groves2019}]\label{lem: Eq noeth characterization }The
following are equivalent for a family $\mathcal{C}$ of groups: 
\begin{itemize}
\item The family $\mathcal{C}$ is equationally noetherian,
\item Every residually $\mathcal{C}$ marked group is finitely presentable
as a residually $\mathcal{C}$ group: for every residually $\mathcal{C}$
marked group $(G,S)$, there is a finitely presented marked group
$(H,S')$ with 
\[
(G,S)=\mathrm{Res}_{\mathcal{C}}(H,S').
\]
\item Sequences of marked quotients of residually $\mathcal{C}$ groups
stabilize: for every sequence $((G_{n},S_{n}))_{n\in\mathbb{N}}$
of residually $\mathcal{C}$ marked groups with $(G_{n},S_{n})\rightarrow(G_{n+1},S_{n+1})$
for all $n$, there is some $n_{0}$ with $(G_{n_{0}},S_{n_{0}})=(G_{n_{0}+1},S_{n_{0}+1})$. 
\end{itemize}
\end{lem}

The following lemma is a straightforward generalization of \cite[Theorem 2.1(5)]{Houcine2007}
to equationally noetherian families. 
\begin{lem}
\label{lem:Eq noeth -> Res C is Qvar}Let $\mathcal{C}$ be a set
of finitely generated groups. If $\mathcal{C}$ is equationally noetherian,
then $\mathcal{RC}$ is a quasivariety.
\end{lem}

It follows that, on equationally noetherian families, the operator
$\mathcal{C}\mapsto\mathcal{RC}$, which maps a class $\mathcal{C}$
of groups to the class of residually $\mathcal{C}$ groups, agrees
with the operator $\mathcal{C}\mapsto\mathrm{QVar}\mathcal{C}$, which
maps a class of groups to the quasivariety it generates. 
\begin{proof}
[Proof of Lemma \ref{lem:Eq noeth -> Res C is Qvar}]We have that
$\mathcal{RC}$ is not a quasivariety if and only if there is an infinite
set of relations $\{r_{i},i\in\mathbb{N}\}$ which implies a relation
$w=1$ in all groups in $\mathcal{C}$, and yet no finite subset of
$\{r_{i},i\in\mathbb{N}\}$ implies $w=1$ in all groups in $\mathcal{C}$.
But in this case the sequence of morphisms 
\[
\mathrm{Res}_{\mathcal{C}}\langle S\mid r_{0}\rangle\rightarrow\mathrm{Res}_{\mathcal{C}}\langle S\mid r_{0},r_{1}\rangle\rightarrow...
\]
 cannot stabilize. Indeed, if it stabilized after $n$ steps, it would
follow that $\mathrm{Res}_{\mathcal{C}}\langle S\mid r_{0},\,i\le n\rangle$
already satisfies all the relations of $\{r_{i},i\in\mathbb{N}\}$,
and thus that $w=1$ in this group, yet by hypothesis this is not
the case.
\end{proof}

\section{Representations and the Weihrauch lattice }

In the following, we mostly follow the surveys \cite{Brattka2021a,Schroeder2021},
further references and historical remarks can also be found there. 

\subsection{Admissible representations, continuous multi-functions}

A \emph{representation} of a set $X$ is a partial surjection $\rho:\subseteq\mathbb{N}^{\mathbb{N}}\rightarrow X$. 

If $\rho(u)=x$, then $u$ is called a $\rho$\emph{-name} of $x$.

The \emph{final topology of a representation} $\rho:\subseteq\mathbb{N}^{\mathbb{N}}\rightarrow X$
is the topology on $X$ given by
\[
O\mathrm{\text{ open}}\iff\rho^{-1}(O)\mathrm{\text{ is open in dom}}(\rho),
\]
where $\mathrm{dom}(\rho)$ is equipped with the subset topology from
Baire space. 

A \emph{multi-function} between sets $X$ and $Y$ is a relation $R\subseteq X\times Y$
which we treat as a function $f$ mapping points of $X$ to subsets
of $Y$ via the formula $f(x)=\{y\in Y\mid(x,y)\in R\}$. We put $\mathrm{dom(}f)=\{x\in X\mid f(x)\neq\emptyset\}$,
and consider that a multi-function $f$ between $X$ and $Y$ is \emph{total}
if $\mathrm{dom(}f)=X$, and that it is \emph{partial} otherwise. 

The fact that $f$ is a total multi-function from $X$ to $Y$ is
denoted via $f\colon X\rightrightarrows Y$, and we denote a partial
multi-function by $f:\subseteq X\rightrightarrows Y$.
\begin{defn}
Let $(X,\rho)$ and $(Y,\tau)$ be represented spaces, and let $f:\subseteq X\rightrightarrows Y$
be a partial multi-function. A \emph{realizer of $f$ with respect
to $\rho$ and $\tau$, }or \emph{$(\rho,\tau)$-realizer of $f$,}
is a function $F:\subseteq\mathbb{N}^{\mathbb{N}}\rightarrow\mathbb{N}^{\mathbb{N}}$
which satisfies the following:
\[
\forall u\in\mathrm{dom}(\rho)\cap\rho^{-1}(\mathrm{dom}(f)),\,\tau(F(u))\in f(\rho(u)).
\]
\end{defn}

Note that the realizer of a multi-function is a single valued function. 

The following diagram illustrates the notion of realizer:

\begin{center}
\begin{tikzcd} [scale=2]
X \arrow[r, "f"] & Y \\ 
\mathbb{N}^\mathbb{N} \arrow[u, "\rho"] \arrow[r, "F" swap] & \mathbb{N}^\mathbb{N} \arrow[u, "\tau"] 
\end{tikzcd}
\end{center}

Let $\rho$ and $\tau$ be representations of a set $X$. Say that
$\rho$ (\emph{continuously) translates} to $\tau$ if the identity
on $X$ has a continuous $(\rho,\tau)$-realizer, i.e. if there is
a continuous partial function $F:\subseteq\mathbb{N}^{\mathbb{N}}\rightarrow\mathbb{N}^{\mathbb{N}}$
defined at least on $\mathrm{dom}(\rho)$ and such that for any $u$
in $\mathrm{dom}(\rho)$, if $u$ is a $\rho$-name of a point $x$
in $X$, then $F(u)$ is a $\tau$-name of $x$.

We denote the fact that $\rho$ translates to $\tau$ by $\rho\le\tau$.
If $\rho\le\tau$ and $\tau\le\rho$, then $\rho$ and $\tau$ are
called (\emph{continuously) equivalent}. When studying the continuous
Weihrauch degrees of functions between represented spaces, representations
can be considered up to equivalence. 
\begin{defn}
[\cite{Kreitz1985,Schroeder2002}]Let $X$ be a topological space.
A representation $\rho:\subseteq\mathbb{N}^{\mathbb{N}}\rightarrow X$
is called \emph{admissible} \emph{for} $X$ if the following hold: 
\begin{enumerate}
\item The representation $\rho$ is continuous, 
\item For every continuous representation $\tau:\subseteq\mathbb{N}^{\mathbb{N}}\rightarrow X$
we have $\tau\le\rho$. 
\end{enumerate}
\end{defn}

It is possible for a representation $\rho$ of a set $X$ to be admissible
for a topology $\mathcal{T}$ different from its final topology. But
this arises only if this topology $\mathcal{T}$ is not sequential
while having as sequentialization the final topology of $\rho$ \cite{Schroeder2002,Schroeder2021},
and in all cases that we will consider, represented spaces are systematically
equipped with the final topologies of their representations. 

\emph{An admissible representation} is a representation admissible
with respect to its final topology. If $\rho$ is admissible for $X$,
$(X,\rho)$ is an \emph{admissibly represented space. }

A space that admits an admissible representation is necessarily $\mathrm{T}_{0}$
(non $\mathrm{T}_{0}$ spaces can be handled by using \emph{multi-representations}
\cite{Schroeder2002}). 
\begin{lem}
[\cite{Kreitz1985,Schroeder2002}]\label{lem:Admissibility C0 and C0 realizable }Let
$(X,\rho)$ be a represented space and let $(Y,\tau)$ be an admissibly
represented space. Let $f\colon X\rightarrow Y$ be a function. Then
the following are equivalent: 
\begin{enumerate}
\item The function $f$ is continuous with respect to the final topologies
of $\rho$ and $\tau$.
\item The function $f$ is admits a continuous realizer. 
\end{enumerate}
\end{lem}

The above lemma permits the definition of a robust notion of continuity
for multi-functions between topological spaces that admit admissible
representations. 
\begin{defn}
A multi-function $f:\subseteq X\rightrightarrows Y$ between topological
spaces $X$ and $Y$ is called \emph{continuous} if it admits a continuous
realizer with respect to some (equiv. any) admissible representations
on $X$ and $Y$. 
\end{defn}

Lemma \ref{lem:Admissibility C0 and C0 realizable } implies that
if $f$ is a single valued function, the above definition does coincide
with continuity. It does not seem to be possible to define the above
notion of continuity for multi-functions without going through the
category of represented spaces. This is discussed in \cite{Brattka1994}.

It is in fact very easy to construct admissible representations of
second countable spaces. 
\begin{defn}
[Kreitz-Weihrauch, \cite{Kreitz1985}]\label{def:Kreitz-Weihrauch, Standard Rep }Let
$X$ be a $\mathrm{T}_{0}$ second countable topological space and
$(B_{i})_{i\in\mathbb{N}}$ a basis for $X$. The \emph{standard representation
associated to} $(B_{i})_{i\in\mathbb{N}}$ is given by 
\[
\forall u\in\mathbb{N}^{\mathbb{N}},\,\rho(u)=x\iff\mathrm{Im}(u)=\{n\in\mathbb{N}\mid x\in B_{n}\}.
\]
\end{defn}

Each numbered basis of $X$ yields a standard representation. However,
the choice of a numbered basis is inconsequential: 
\begin{thm}
[\cite{Kreitz1985}]All standard representations of a second countable
space are admissible and continuously equivalent. Every admissible
representation of a second countable space is continuously equivalent
to a standard representation. 
\end{thm}

\subsection{\label{subsec:Some-constructions-of RPZ}Some constructions of representations }

\subsubsection{Representation of a product }

If $(X,\rho)$ and $(Y,\tau)$ are represented spaces, the set $X\times Y$
is naturally equipped with the representation $\rho\times\tau$, given
by 
\[
(\rho\times\tau)((u_{n})_{n\in\mathbb{N}})=(\rho((u_{2n})_{n\in\mathbb{N}}),\tau((u_{2n+1})_{n\in\mathbb{N}})).
\]
The domain of $\rho\times\tau$ is the set of sequences $(u_{n})_{n\in\mathbb{N}}$
such that $(u_{2n})_{n\in\mathbb{N}}\in\mathrm{dom}(\rho)$ and $(u_{2n+1})_{n\in\mathbb{N}}\in\mathrm{dom}(\tau)$. 

\subsubsection{\label{subsec:Representation-of-an infinite product}Representation
of an infinite product }

Let $(n,k)\mapsto\langle n,k\rangle$ be the standard Cantor pairing
function defined by $\langle n,k\rangle:=\frac{1}{2}(n+k+1)(n+k)+k.$ 

We define a pairing function $\langle\cdot\rangle:(\mathbb{N}^{\mathbb{N}})^{\mathbb{N}}\to\mathbb{N}^{\mathbb{N}}$
by $\langle p_{0},p_{1},p_{2},\ldots\rangle\langle n,k\rangle:=p_{n}(k)$
for all $p_{i}\in\mathbb{N}^{\mathbb{N}}$ and $n,k\in\mathbb{N}$. 

If $(X,\rho)$ is a represented space, we define a representation
$\delta_{X^{\mathbb{N}}}$ of $X^{\mathbb{N}}$ by:
\[
\delta_{X^{\mathbb{N}}}(\langle p_{0},p_{1},p_{2},\ldots\rangle)=(\rho(p_{n}))_{n\in\mathbb{N}}.
\]

\subsubsection{Sierpi\'{n}ski space}

The Sierpi\'{n}ski space $\mathbb{S}$ is the set $\{0,1\}$ equipped
with the topology $\{\emptyset,\{1\},\{0,1\}\}$. It admits an admissible
representation $c_{\mathbb{S}}:\mathbb{N}^{\mathbb{N}}\rightarrow\mathbb{S}$
given by:
\[
c_{\mathbb{S}}(0^{\omega})=0,
\]
\[
c_{\mathbb{S}}(u)=1\text{ for any }\ensuremath{u\neq0^{\omega}}.
\]

\subsubsection{Restrictions}

If $(X,\rho)$ is a represented space and $Y\subseteq X$, we let
$\rho_{\vert Y}$ be the restriction of $\rho$ to $\rho^{-1}(Y)$.
If $X$ is a second countable space and $\rho$ is admissible, then
the final topology of $\rho_{\vert Y}$ is the subset topology on
$Y$ and $\rho_{\vert Y}$ is also admissible. 

\subsubsection{\label{subsec:Jump-of-a-RPZ space}Jump of a represented space }

Let $(X,\delta)$ be a represented space. We define a new representation
$\delta':\subseteq\mathbb{N}^{\mathbb{N}}\rightarrow X$ of $X$ by
the following:
\begin{align*}
\delta'(p)=x\iff & \begin{cases}
p\text{, seen as an element of }(\mathbb{N}^{\mathbb{N}})^{\mathbb{N}}\text{ via a pairing function},\\
\text{is a sequence that converges to a \ensuremath{\delta}-name of \ensuremath{x}}
\end{cases}
\end{align*}

The represented space $(X,\delta')$ is called the \emph{jump} of
$(X,\delta)$. Obviously $\delta\le\delta'$. 

Note that the final topology of $\delta'$ is always the indiscrete
topology. 

\subsection{Continuous Weihrauch reduction}
\begin{defn}
A multi-function between represented spaces is called a \textbf{problem}. 

Let $P_{1}:\subseteq X\rightrightarrows Y$ and $P_{2}:\subseteq X\rightrightarrows Y$
be problems. We say that $P_{1}$ \textbf{solves} $P_{2}$, (sometimes
also: $P_{1}$ is a strengthening of $P_{2}$ and $P_{2}$ a weakening
of $P_{1}$) and write $P_{1}\sqsubseteq P_{2}$, if 
\[
\mathrm{dom}(P_{2})\subseteq\mathrm{dom}(P_{1})\,\&\,\forall x\in\mathrm{dom}(P_{2}),\,P_{1}(x)\subseteq P_{2}(x).
\]
\end{defn}

When a multi-function $P$ is interpreted as a problem, we see $P(x)$
as a set of valid outputs on input $x$. And thus if another problem
$Q$ satisfies $Q\sqsubseteq P$, it means that each valid output
of $Q$ is a valid output of $P$, hence the term ``$Q$ solves $P$''. 

If $P:\subseteq X\rightrightarrows Y$ and $Q:\subseteq Z\rightrightarrows W$
are problems, their \textbf{product}, denoted by $P\times Q$, is
the problem 
\begin{align*}
P\times Q:\subseteq X\times Z & \rightrightarrows Y\times W\\
(u,v) & \mapsto P(u)\times Q(v).
\end{align*}

We can now define the topological version Weihrauch reduction. 
\begin{defn}
[Continuous Weihrauch reduction] Let $P:\subseteq X\rightrightarrows Y$
and $Q:\subseteq Z\rightrightarrows W$ be problems. We say that $P$
\textbf{(continuously) Weihrauch reduces} to $Q$, and write $P\ge_{W}Q$,
if there is a represented space $V$ and continuously realizable multi-functions
$h:\subseteq V\times W\rightrightarrows Y$ and $k:\subseteq X\rightrightarrows V\times Z$
such that 
\[
h\circ(\mathrm{id}_{V}\times P)\circ k\sqsubseteq Q.
\]
\end{defn}

\begin{figure}
	\centering 		\begin{tikzpicture}[>=stealth, node distance=2cm] 		
			\filldraw[gray!20, draw=black] (1,-1.5) rectangle (10,1.5); 			\node at (1.7,1.2) {Q}; 		
			\node (p) at (0,0) {x}; 			
			\coordinate(secret) at (1.3,0) {}; 		
	\node[circle,draw,minimum size=1cm] (K) [right=of p] {k}; 		
	\node[draw,minimum width=1cm,minimum height=1cm] (G) [right=of K] {P}; 		
	\node[circle,draw,minimum size=1cm] (H) [right=of G] {h}; 		
	\node (Fp) [right=of H] {Q(x)}; 		
			\draw[->, line width=0.8pt] (p) -- (K); 		
	\draw[->, line width=0.8pt] (K) -- (G); 		
	\draw[->, line width=0.8pt] (G) -- (H); 	
		\draw[->, line width=0.8pt] (H) -- (Fp); 	
				\draw[->, line width=0.8pt] (K) .. controls +(0,-1) and +(0,-1) ..	(H.south west);				
		\end{tikzpicture} 	
		\caption{Weihrauch reducibility}
\end{figure}
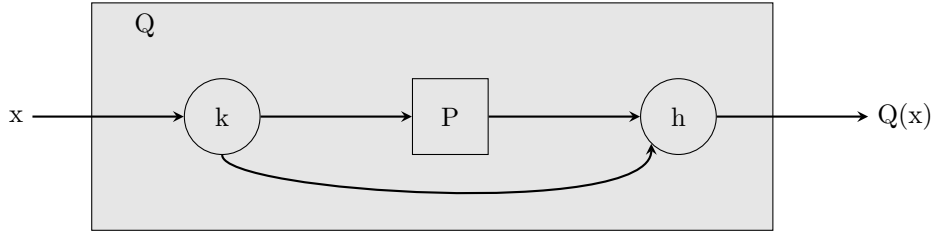

\subsection{Comments on the definition of Weihrauch reducibility }

We want to illustrate the pros of allowing continuous multi-functions
in the definition of Weihrauch reducibility in order to capture the
idea of ``a function being more discontinuous than another''.
\begin{defn}
We denote by $\mathrm{LPO}$ (for \emph{Limited Principle of Omniscience}\footnote{The name comes from constructive mathematics. The constructivist Errett
Bishop was the first to talk about Omniscience principles.}) the Weihrauch equivalence class of the following problem: 

\[
\mathrm{LPO}:\mathbb{N}^{\mathbb{N}}\rightarrow\{0,1\},\,\mathrm{LPO}(u)\mapsto\begin{cases}
1, & \text{ if }u=0^{\omega},\\
0, & \text{otherwise.}
\end{cases}
\]
\end{defn}

LPO is the equivalence class of all problems whose discontinuity points
are isolated, and it is the easiest discontinuous single valued problem.
The following result was obtained by Weihrauch \cite{Weihrauch1992}
in the context of admissible representations of metric spaces, but
it is easy to generalize it to $\mathrm{T}_{0}$ second countable
spaces. 
\begin{prop}
A function $P:X\rightarrow Y$ between $\mathrm{T}_{0}$ second countable
spaces is discontinuous if and only if $P\ge_{W}\mathrm{LPO}$. 
\end{prop}

\begin{example}
The following problems are all instances of LPO, however, they would
not be equivalent if we did not allow for multi-functions in the definition
of Weihrauch reducibility:
\begin{itemize}
\item $f\colon\mathbb{R}\rightarrow\mathbb{R},\,x\mapsto\begin{cases}
0, & \text{if }x<1/3,\\
1, & \text{otherwise;}
\end{cases}$
\item $\tilde{f}$, the restriction of $f$ to the Cantor ternary set; 
\item the injection $P:\mathbb{S}\rightarrow\{0,1\}$ of the Sierpi\'{n}ski
space into $\{0,1\}$ equipped with the discrete topology. 
\end{itemize}
Indeed, any continuous function $\mathbb{S}\rightarrow\mathbb{R}$
is constant because $\mathbb{R}$ is $\mathrm{T}_{2}$, and similarly
any continuous function from $\mathbb{R}$ to the Cantor ternary set
is constant because $\mathbb{R}$ is connected. On the other hand,
many non-constant continuous multi-functions exist between these spaces.
For instance, if $(u_{n})_{n\in\mathbb{N}}$ is a sequence of real
numbers converging to $l\in\mathbb{R}$, then the multi-function $\mathbb{S}\rightrightarrows\mathbb{R}$
mapping the open point to $\{u_{n},\,n\in\mathbb{N}\}$ and the closed
point to $l$ is continuous. We leave it to the reader to explicitly
write down Weihrauch reductions between the above problems. 
\end{example}

Here we give an example that illustrates the need to pass through
the input. 
\begin{example}
Consider the following functions:
\begin{itemize}
\item $f\colon\mathbb{R}\rightarrow\mathbb{R},\,x\mapsto\begin{cases}
0, & \text{if }x<0,\\
1, & \text{otherwise;}
\end{cases}$
\item $\hat{f}:\mathbb{R}\rightarrow\mathbb{R},\,x\mapsto\begin{cases}
x, & \text{if }x<0,\\
x+1, & \text{otherwise.}
\end{cases}$ 
\end{itemize}
One easily checks that both problems are equivalent to LPO. However,
by using a definition of Weihrauch reducibility that does not include
the $\mathrm{id}_{V}$ component in $h\circ(\mathrm{id}_{V}\times P)\circ k\sqsubseteq Q$,
those problems become non-equivalent. 

The reduction obtained by removing the identity in the definition
of Weihrauch reducibility is called strong Weihrauch reducibility
\cite{Brattka2021a}.
\end{example}

\subsection{Some classical problems}

We introduce some benchmark problems. 

The problem \emph{Limit on the natural numbers,} denoted by $\mathrm{Lim}_{\mathbb{N}}:\subseteq\mathbb{N}^{\mathbb{N}}\rightarrow\mathbb{N}$,
is defined on the subset of $\mathbb{N}^{\mathbb{N}}$ consisting
of converging sequences, a sequence is mapped to its limit. 

The problem \emph{Limit on Baire Space,} $\mathrm{Lim}:\subseteq\mathbb{N}^{\mathbb{N}}\rightarrow\mathbb{N}^{\mathbb{N}}$
is defined thanks to a bijection between $\mathbb{N}^{\mathbb{N}}$
and $(\mathbb{N}^{\mathbb{N}})^{\mathbb{N}}$: $\mathrm{Lim}$ is
defined on the set of elements of $\mathbb{N}^{\mathbb{N}}$ which,
seen as elements of $(\mathbb{N}^{\mathbb{N}})^{\mathbb{N}}$, define
a converging sequence, and then a sequence is mapped to its limit. 

The \emph{jump of a problem} $P:\subseteq X\rightrightarrows Y$ is
the problem obtained by replacing the representation of $X$ by its
jump (Section \ref{subsec:Jump-of-a-RPZ space}). Thus extensionally
it is the same multi-function, but the representation has been changed
on the input side.

The $n$-th jump of a problem $P$ is denoted by $P^{(n)}$. 

A function $f:\subseteq\mathbb{N}^{\mathbb{N}}\rightarrow\mathbb{N}^{\mathbb{N}}$
is called $\Sigma_{n}^{0}$-measurable if the preimage of any open
set is in $\Sigma_{n}^{0}$. A function between represented spaces
is called $\Sigma_{n}^{0}$-measurable if it admits a $\Sigma_{n}^{0}$-measurable
realizer.

A problem $P$ is $\Sigma_{n}^{0}$-\emph{measurable complete} if
for every problem $Q$ that is $\Sigma_{n}^{0}$-measurable we have
$P\ge_{W}Q$. 
\begin{thm}
[\cite{Brattka2004}]\label{thm: Lim^(n) is measurable complete}$\mathrm{Lim}$
is $\Sigma_{2}^{0}$-measurable complete, and more generally $\mathrm{Lim}^{(n)}$
is $\Sigma_{n+2}^{0}$-measurable complete. 
\end{thm}

\begin{example}
\label{exa:LPO' Exists infty}The problem $\mathrm{LPO}'$, the jump
of $\mathrm{LPO}$, is easily seen to be equivalent to the problem
$\exists^{\infty}:\{0,1\}^{\mathbb{N}}\rightarrow\{0,1\}$, given
by 
\[
\exists^{\infty}((u_{n})_{n\in\mathbb{N}})=1\iff\exists^{\infty}n,\,u_{n}=1.
\]
See for instance \cite{Pauly2024}. 
\end{example}

\subsection{Finite parallelizability }

For problems $P:\subseteq X\rightrightarrows Y$ and $Q:\subseteq Z\rightrightarrows W$,
recall that their product $P\times Q:\subseteq X\times Z\rightrightarrows Y\times W$
is given by 
\[
\left(P\times Q\right)(x,z)=P(x)\times Q(z),
\]
with $\mathrm{dom}(P\times Q)=\mathrm{dom}(P)\times\mathrm{dom}(Q)$.

The \emph{finite parallelization} of a problem $P:\subseteq X\rightrightarrows Y$
is the problem 
\begin{align*}
P^{*}:\subseteq\bigcup_{n\in\mathbb{N}}\{n\}\times X^{n} & \rightrightarrows\bigcup_{n\in\mathbb{N}}\{n\}\times Y^{n}\\
(i,x_{1},...,x_{i}) & \mapsto\{i\}\times P(x_{1})\times...\times P(x_{i})
\end{align*}
with domain $\bigcup_{n\in\mathbb{N}}\{n\}\times\mathrm{dom}(X)^{n}$. 

In words: the instances of $P^{*}$ consist in finitely many instances
of $P$, and solutions for these instances are simply the tuples consisting
of solutions for each instances.

A problem $P$ is \emph{finitely parallelizable} if $P\equiv_{W}P^{*}$.
This is easily seen to be equivalent to the fact that $P\equiv_{W}P\times P$. 

The problem $\mathrm{Res}_{\mathcal{C}}$ is always finitely parallelizable
(Proposition \ref{prop: Finite parallelization}). 

\subsection{Infinite parallelization}

Let $P:\subseteq X\rightrightarrows Y$ be a problem. The \emph{infinite
parallelization} of $P$ is the problem $\hat{P}:\subseteq X^{\mathbb{N}}\rightrightarrows Y^{\mathbb{N}}$
given by 
\[
\hat{P}((u_{n})_{n\in\mathbb{N}})=(P(u_{n}))_{n\in\mathbb{N}}.
\]
Here $X^{\mathbb{N}}$ and $Y^{\mathbb{N}}$ are equipped with the
representations induced by those of $X$ and $Y$ as defined in Section
\ref{subsec:Representation-of-an infinite product}.

We will use the following equalities \cite{Brattka2021a}: 
\begin{prop}
\label{prop:Lim =00003D LPO parallalized}
\[
\mathrm{Lim}\equiv_{W}\widehat{\mathrm{LPO}},
\]
\[
\mathrm{Lim}'\equiv_{W}\widehat{\mathrm{LPO}'}.
\]
\end{prop}

\section{\label{sec:Discretization-problem-associate Alexandrov Top}Discretization
problem associated to the Alexandrov topology of an ordinal}

\subsection{Rank function on a well-founded partial order }

The following was used by Hertling \cite{Hertling2020} to study continuous
Weihrauch degrees.
\begin{defn}
\label{def:Ordinal_Rank}Let $(A,\le)$ be a well founded order. We
define a function $\mathrm{rank}_{(A,\le)}:A\rightarrow\mathrm{ORD}$,
where $\mathrm{ORD}$ is the class of ordinals, whose order we designate
by $\le$, inductively as follows: 
\[
\mathrm{rank}_{(A,\le)}(a)=\sup\{\mathrm{rank}_{(A,\le)}(b)+1\mid b<a\}
\]
for any $a\in A$. (The base case is given by $\sup\emptyset=0$.) 

We define the rank of $(A,\le)$ as 
\[
\mathrm{rank}((A,\le))=\sup\{\mathrm{rank}_{(A,\le)}(a)+1\mid a\in A\}.
\]
\end{defn}

\begin{example}
The rank of an ordinal is that ordinal. The rank of the lattice of
marked quotients of $(\mathbb{Z}^{2},((1,0),(0,1)))$ is $\omega2+1$. 
\end{example}

\subsection{Alexandrov topology on a partially ordered set}

If $(X,\le)$ is a partially ordered set, we define a $\mathrm{T}_{0}$
topology on it by 
\[
U\text{ is open}\iff U\text{ is an upper set},
\]
i.e. $U$ is open iff $\forall x\in U,\forall y\in X,x\le y\implies y\in U$.
This defines \emph{the Alexandrov topology} associated to $\le$ \cite{Arenas1999}.
This topology has the property that arbitrary intersections of open
sets are open, and, in fact, any $\mathrm{T}_{0}$ topology with the
property that arbitrary intersections of open sets are open is the
Alexandrov topology of its specialization order. 

And we have the obvious properties:
\begin{prop}
\label{prop:order-preserving iff C0 alexandrov}The order preserving
maps between partially ordered sets coincide with the continuous functions
with respect to their associated Alexandrov topologies. \qed
\begin{prop}
\label{prop: ordinal-rank order is CB rank alexandrov top}The ordinal
rank of a well-founded order $(A,\le)$ is identical to the Cantor-Bendixson
rank of the Alexandrov topology induced by $(A,\le^{\mathrm{op}})$.
\qed
\end{prop}

\end{prop}

The following proposition will be useful to relate the ordinal rank
to Weihrauch reductions. 

A multi-function $f$ is called injective if $\forall x,y\in\mathrm{dom}(f),x\ne y\implies f(x)\cap f(y)=\emptyset$. 
\begin{prop}
\label{prop: ordinal rank and multi-injection}Let $(A,\le)$ be a
countable well founded order, and let $\alpha$ be its ordinal rank.
Equip $A$ and $\alpha$ with the Alexandrov topologies associated
to the respective reverse orders $\le^{\mathrm{op}}$. Then the function
$\mathrm{rank}_{(A,\le)}:A\rightarrow\alpha$ is continuous, and its
multi-inverse $\mathrm{rank}_{(A,\le)}^{-1}:\alpha\rightrightarrows A$
is an injective continuous multi-function. 
\end{prop}

\begin{proof}
The function $\mathrm{rank}_{(A,\le)}:A\rightarrow\alpha$ is order
preserving, thus continuous. It is onto by definition of $\alpha$. 

Its multi-inverse $\mathrm{rank}_{(A,\le)}^{-1}:\alpha\rightrightarrows A$
is clearly injective. 

We show that $\mathrm{rank}_{(A,\le)}^{-1}$ has a continuous realizer. 

The Alexandrov topology on $(A,\le^{\mathrm{op}})$ is second countable,
a countable basis being given by $(B_{a})_{a\in A}$, where $B_{a}=\{b\in A\mid b\le a\}$.
It follows that the standard representation (Definition \ref{def:Kreitz-Weihrauch, Standard Rep })
associated to the Alexandrov topology on $(A,\le^{\mathrm{op}})$
is given by: a point $a$ is represented by a weakly descending sequence
in $A$ which is eventually constant at $a$.

Consider a well ordering $\succeq$ on $A$ (distinct from $\le$
which need not be total). 

Then we define a continuous realizer of $\mathrm{rank}_{(A,\le)}^{-1}$
by the following: 

A weakly descending sequence $(\beta_{n})\in\alpha^{\mathbb{N}}$
is mapped to a weakly descending sequence $(a_{n})\in A^{\mathbb{N}}$
given by: 

\[
a_{0}=\inf_{\succeq}\{a\in A\mid\mathrm{rank}_{(A,\le)}(a)=\beta_{0}\}
\]
\[
a_{n+1}=\inf_{\succeq}\{a\in A\mid\mathrm{rank}_{(A,\le)}(a)=\beta_{n+1}\,\&\,a\le a_{n}\}.
\]
This clearly defines a continuous realizer to $\mathrm{rank}_{(A,\le)}^{-1}:\alpha\rightrightarrows A$.
\end{proof}

\subsection{\label{subsec:Discretization-problem}Discretization problem}

The following notion will be useful. Let $(X,\rho)$ be a countable
represented space. 
\begin{lem}
Let $f_{1}:X\rightarrow\mathbb{N}$ and $f_{2}:X\rightarrow\mathbb{N}$
be injections. Then $f_{1}\equiv_{W}f_{2}$. 
\end{lem}

\begin{proof}
Straightforward. 
\end{proof}
\begin{defn}
The \emph{discretization problem of a countable represented space
$(X,\rho)$, }denoted\emph{ $\mathrm{Dscr}((X,\rho))$}, is the equivalence
class of the problem $f\colon X\rightarrow\mathbb{N}$, for any injection
$f$. 
\end{defn}

And:
\begin{defn}
The \emph{discretization problem of a countable topological space
$X$, }denoted \emph{$\mathrm{Dscr}(X)$,} is defined as $\mathrm{Dscr}((X,\rho))$
for any admissible representation $\rho$ of $X$. 
\end{defn}

\subsection{Weakly descending sequence in an ordinal }

Let $\alpha$ be a countable ordinal. 

Fix a surjection $f\colon\mathbb{N}\rightarrow\alpha$. We consider
the following problem:
\[
\mathrm{Lim}_{\searrow(\alpha,f)}:\subseteq\mathbb{N}^{\mathbb{N}}\rightrightarrows\mathbb{N},
\]
\[
\mathrm{dom}(\mathrm{Lim}_{\searrow(\alpha,f)})=\{(u_{n})_{n\in\mathbb{N}}\in\mathbb{N}^{\mathbb{N}}\mid\forall n\in\mathbb{N},f(u_{n+1})\le f(u_{n})\},
\]
\[
\forall(u_{n})_{n\in\mathbb{N}}\in\mathrm{dom}(\mathrm{Lim}_{\searrow(\alpha,f)}),\mathrm{Lim}_{\searrow(\alpha,f)}((u_{n})_{n\in\mathbb{N}})=\{t\in\mathbb{N}\mid f(t)=\inf_{n\in\mathbb{N}}(f(u_{n}))\}.
\]
Thus the input to $\mathrm{Lim}_{\searrow(\alpha,f)}$ is a weakly
descending sequence in $\alpha$, and the output is its limit. The
above problem in fact does not depend on the chosen surjection $f$. 
\begin{lem}
If $f$ and $g$ are two surjections $f,g:\mathbb{N}\rightarrow\alpha$,
then $\mathrm{Lim}_{\searrow(\alpha,f)}\equiv_{W}\mathrm{Lim}_{\searrow(\alpha,g)}$. 
\end{lem}

\begin{proof}
Left to the reader. 
\end{proof}
\begin{defn}
Define $\mathrm{Lim}_{\searrow\alpha}$ to be the equivalence class
of $\mathrm{Lim}_{\searrow(\alpha,f)}$ for any surjection $f\colon\mathbb{N}\rightarrow\alpha$. 
\end{defn}

\begin{lem}
The map $\alpha\mapsto\mathrm{Lim}_{\searrow\alpha}$ is an order
preserving injection from the set of countable ordinals to the continuous
Weihrauch degrees. 
\end{lem}

\begin{proof}
That $\alpha\mapsto\mathrm{Lim}_{\searrow\alpha}$ is order preserving
is straightforward. Injectivity follows directly from Hertling's notion
of level of discontinuity of a problem \cite{Hertling1996a}, which
is an ordinal invariant similar to the Cantor-Bendixson rank for functions,
and from the obvious fact that $\mathrm{Lim}_{\searrow\alpha}$ has
level $\alpha$. 
\end{proof}
\begin{example}
The problem $\mathrm{Lim}_{\searrow2}$ is equivalent to $\mathrm{LPO}$.
The problem $\mathrm{Lim}_{\searrow\omega}$ is $\mathrm{LPO}^{*}$,
the finite parallelization of $\mathrm{LPO}$. 
\end{example}

Notice the following easy proposition:
\begin{prop}
\label{prop: limit (ordinal problem) =00003D problem (limit ordinal)}Let
$\alpha$ be a countable limit ordinal. Then for any problem $P$,
$P\ge_{W}\mathrm{Lim}_{\searrow\alpha}$ if and only if $P\ge_{W}\mathrm{Lim}_{\searrow\beta}$
for every $\beta<\alpha$. 
\end{prop}

\begin{proof}
The direct implication is straightforward. For the converse, notice
that an instance of $\mathrm{Lim}_{\searrow\alpha}$ is a weakly descending
sequence whose first term necessarily belongs to some $\beta<\alpha$.
Fix the functions $k_{\beta}$ and $h_{\beta}$ that witness the reduction
$P\ge_{W}\mathrm{Lim}_{\searrow\beta}$ for each ordinal $\beta\in\alpha$.
The total reduction for $P\ge_{W}\mathrm{Lim}_{\searrow\alpha}$ is
then: given $(u_{n})$, a weakly descending sequence in $\alpha$,
let $\beta$ be $u_{0}+1$, apply the reduction $k_{\beta}$ followed
by $P$ followed by $h_{\beta}$, this gives the limit of the sequence
$(u_{n})$. 
\end{proof}

\subsection{Discretization problems and limit in a countable ordinal. }

The problems $\mathrm{Lim}_{\searrow\alpha}$ are related to discretization
problems: 
\begin{prop}
The problem $\mathrm{Lim}_{\searrow\alpha}$ is the discretization
problem associated to the Alexandrov topology on $\alpha^{\mathrm{op}}$. 
\end{prop}

\begin{proof}
It suffices to note that the description of an element of $\alpha$
by a weakly descending sequence in $\alpha$ is exactly the standard
Kreitz-Weihrauch representation associated to the Alexandrov topology
on $\alpha^{\mathrm{op}}$ (see Definition \ref{def:Kreitz-Weihrauch, Standard Rep }). 
\end{proof}
Note also: 
\begin{prop}
\label{prop: CB rank upper bound discretization }Let $(X,\rho)$
be a countable represented space without condensation, and let $\alpha$
be its Cantor-Bendixson rank. Then \emph{$\mathrm{Dscr}((X,\rho))\le_{W}\mathrm{Lim}_{\searrow\alpha}$.}
\end{prop}

\begin{proof}
Straightforward. 
\end{proof}

\section{\label{sec:Different-representations-of marked groups}Different
representations of marked groups}

We introduce different representations and their associated final
topologies. 

\subsection{Space of marked groups }

As explained in Section \ref{sec:Marked-groups}, a $k$-marked group
can be seen as a normal subgroup of $\mathbb{F}_{k}$. For each $k$,
we fix a bijection $\theta_{k}:\mathbb{N}\rightarrow\mathbb{F}_{k}$.
We can then identify the power set of $\mathbb{F}_{k}$ with $\{0,1\}^{\mathbb{N}}$.
The set of $k$-marked groups can thus be seen as a subset of $\{0,1\}^{\mathbb{N}}$.
By equipping $\{0,1\}^{\mathbb{N}}$ with the Cantor space topology,
the set of $k$-marked groups inherits the subset topology which makes
it a compact Polish space, called\emph{ the space of $k$-marked groups}.
We define \emph{the space of marked groups} by taking the disjoint
union topology of the spaces of $k$-marked groups. 

This topology is associated to the admissible representation $\rho_{\mathrm{WP}}$,
defined by:
\[
\mathrm{dom}(\rho_{\mathrm{WP}})=\{(u_{n})_{n\in\mathbb{N}}\in\{0,1\}^{\mathbb{N}}\mid\{\theta_{u_{0}}(i)\mid u_{i}=1\}\mathrm{\text{ is a normal subgroup of }}\mathbb{F}_{u_{0}}\};
\]
\[
\rho_{\mathrm{WP}}((u_{n})_{n\in\mathbb{N}})=\mathbb{F}_{u_{0}}/\{\theta_{u_{0}}(i)\mid u_{i}=1\}.
\]
In other words: the name of a marked group encodes its word problem.
Hence the subscript WP: $\rho_{\mathrm{WP}}$ is called the \emph{word
problem representation. }

We denote by $\mathcal{G}_{\mathrm{WP}}$ the space of marked groups,
whose topology is associated to the admissible representation $\rho_{\mathrm{WP}}$
. 

\subsection{Scott topology }

Each $k$-marked group can be seen as an element of $\{0,1\}^{\mathbb{N}}$.
If, instead of considering the product topology associated to the
discrete topology on $\{0,1\}$, we take the product topology associated
to the Sierpi\'{n}ski topology on $\{0,1\}$, with $\{1\}$ as only
non-trivial open set, we obtain what is known as \emph{the Scott topology
}on $\{0,1\}^{\mathbb{N}}$. We denote by $\mathbb{S}^{\mathbb{N}}$
the set of binary sequences equipped with the Scott topology. 

A basis for this topology is given by sets of the form $\Omega_{A}=\{(u_{i})_{i\in\mathbb{N}}\in\{0,1\}^{\mathbb{N}}\mid\forall i\in A,\,u_{i}=1\}$
for $A$ a finite subset of $\mathbb{N}$. In terms of marked groups,
the basic open sets have the forms ``those marked groups that satisfy
the relations $\{r_{1},...,r_{n}\}$'', for any finite set of relations
$\{r_{1},...,r_{n}\}$. 

The Scott topology is the one that appears naturally when we consider
the representation where groups are given by presentations. We now
give three equivalent ways of defining this represenation. 

Again a sequence of free groups equipped with bases $(\mathbb{F}_{k},S_{k})$
is fixed, and $\theta_{k}:\mathbb{N}\rightarrow\mathbb{F}_{k}$ is
the short-lex bijection induced by the basis $S_{k}$. 

The \emph{presentation representation} $\rho_{\mathrm{Pres}}$ is
given by 
\begin{align*}
\rho_{\mathrm{Pres}}:\mathbb{N}^{\mathbb{N}} & \rightarrow\mathcal{G}\\
(u_{n})_{n\in\mathbb{N}} & \mapsto\langle S_{u_{0}}\mid\theta_{u_{0}}(u_{i}),i\ge1\rangle.
\end{align*}

The \emph{full presentation representation} $\rho_{\mathrm{FullPres}}$
only uses maximal presentations, i.e. presentations where all the
relations of a marked groups are listed. It is simply the restriction
of the above representation to sequences that define a normal subgroup:
\[
\mathrm{dom}(\rho_{\mathrm{FullPres}})=\{(u_{n})_{n\in\mathbb{N}}\in\mathbb{N}^{\mathbb{N}}\mid\{\theta_{u_{0}}(u_{i}),i\ge1\}\mathrm{\text{ is a normal subgroup of }}\mathbb{F}_{u_{0}}\};
\]
\begin{align*}
\rho_{\mathrm{FullPres}}:\subseteq\mathbb{N}^{\mathbb{N}} & \rightarrow\mathcal{G}\\
(u_{n})_{n\in\mathbb{N}} & \mapsto\langle S_{u_{0}}\mid\theta_{u_{0}}(u_{i}),i\ge1\rangle.
\end{align*}

The \emph{Sierpi\'{n}ski word-problem representation} $\rho_{\mathbb{S}\mathrm{WP}}$
is similar to $\rho_{\mathrm{WP}}$ except that $\{0,1\}$ is equipped
with the Sierpi\'{n}ski topology: the set $\mathcal{G}^{k}$ of $k$-marked
groups is seen as a subset of $\mathbb{S}^{\mathbb{N}}$ and equipped
with the induced topology and representation. A formal definition
is as follows. Let $c_{\mathbb{S}^{\mathbb{N}}}$ be the admissible
representation of $\mathbb{S}^{\mathbb{N}}$, defined as the product
representation of $c_{\mathbb{S}}$ (see Section \ref{subsec:Some-constructions-of RPZ}).
We define $\rho_{\mathbb{S}\mathrm{WP}}^{k}$ to be the restriction
of $c_{\mathbb{S}^{\mathbb{N}}}$ to sequences that define a marked
group. 

We then take $\rho_{\mathbb{S}WP}$ as the disjoint union representation:
\[
\mathrm{dom}(\rho_{\mathbb{S}\mathrm{WP}})=\{(u_{n})_{n\in\mathbb{N}}\in\mathbb{N}^{\mathbb{N}}\mid(u_{n})_{n\ge1}\in\mathrm{dom}(\rho_{\mathbb{S}\mathrm{WP}}^{u_{0}})\};
\]
\[
\forall(u_{n})_{n\in\mathbb{N}}\in\mathrm{dom}(\rho_{\mathbb{S}\mathrm{WP}}),\,\rho_{\mathbb{S}\mathrm{WP}}((u_{n})_{n\in\mathbb{N}})=\rho_{\mathbb{S}\mathrm{WP}}^{u_{0}}((u_{n})_{n\ge1}).
\]

By Lemma \ref{lem:The-representations-Pres and Sier are equiv} and
Corollary \ref{cor: Pres and FullPres} proved below, $\rho_{\mathrm{Pres}}\equiv\rho_{\mathrm{FullPres}}\equiv\rho_{\mathbb{S}\mathrm{WP}}$.
By the following lemma, these three representations are admissible.
\begin{lem}
The representation $\rho_{\mathrm{FullPres}}$ is admissible for the
Scott topology on $\mathcal{G}$. 
\end{lem}

\begin{proof}
The Kreitz-Weihrauch standard representation associated to the natural
basis for the Scott topology on $\mathcal{G}$ is the following: the
name of a group is a list that contains all finite sets of relations
that this group satisfies. But this is clearly equivalent to the representation
$\rho_{\mathrm{FullPres}}$. 
\end{proof}
We denote by $\mathcal{G}_{\mathrm{Pres}}$ the set of marked groups
equipped with the representation $\rho_{\mathrm{Pres}}$ whose final
topology is the Scott topology on the lattice of marked groups. 

\section{Moving between representations}

\subsection{Different representations of $\mathcal{G}_{\mathrm{Pres}}$ }
\begin{lem}
\label{lem:The-representations-Pres and Sier are equiv}The representations
$\rho_{\mathrm{Pres}}$ and $\rho_{\mathbb{S}\mathrm{WP}}$ are equivalent. 
\end{lem}

\begin{proof}
We first show that $\rho_{\mathrm{Pres}}\le\rho_{\mathbb{S}\mathrm{WP}}$. 

Recall that for each $k>0$, $\theta_{k}:\mathbb{N}\rightarrow\mathbb{F}_{k}$
is a bijection between $\mathbb{N}$ and a fixed rank $k$ free group.

Given a presentation $\langle s_{0},...,s_{k}\,\mid\,r_{0},r_{1},r_{2},...\rangle$
for a marked group $(G,S)$, we define a double sequence $(t_{i,j})_{(i,j)\in\mathbb{N}^{2}}\in\{0,1\}^{\mathbb{N}^{2}}$
by the following: $t_{i,j}=1$ if $\theta_{k}(i)$ can be written
as a product of at most $j$ conjugates of the elements $\{r_{0},...,r_{j}\}$,
the conjugators having themselves length at most $j$. Otherwise $t_{i,j}=0$. 

It is clear that $(t_{i,j})_{(i,j)\in\mathbb{N}^{2}}$ depends continuously
on the presentation $\langle s_{0},...,s_{k}\,\mid\,r_{0},r_{1},r_{2},...\rangle$.
Furthermore, $\theta_{k}(i)=1$ in $(G,S)$ if and only if there exists
some $j$ such that $t_{i,j}=1$, and otherwise $t_{i,j}=0$ for all
$j\in\mathbb{N}$. Thus $(t_{i,j})_{(i,j)\in\mathbb{N}^{2}}$ is a
$\rho_{\mathbb{S}\mathrm{WP}}$-name of $(G,S)$. 

We now show that $\rho_{\mathbb{S}\mathrm{WP}}\le\rho_{\mathrm{Pres}}$. 

Consider a marked group $(G,S)$ given by a double sequence $(t_{i,j})_{(i,j)\in\mathbb{N}^{2}}\in\{0,1\}^{\mathbb{N}^{2}}$
that satisfies $\theta_{k}(i)=1$ in $(G,S)$ if and only if $\exists j\in\mathbb{N},\,t_{i,j}=1$.
Consider Cantor's pairing function $(i,j)\mapsto\langle i,j\rangle$.
We define a presentation $\langle s_{0},...,s_{k}\,\mid\,r_{0},r_{1},r_{2},...\rangle$
for $(G,S)$ by $r_{\langle i,j\rangle}=\theta_{k}(i)$ if $t_{i,j}=1$,
otherwise $r_{\langle i,j\rangle}=1_{\mathbb{F}_{k}}$. It is clear
that this indeed defines a presentation of $(G,S)$, and the sequence
$(r_{i})_{i\in\mathbb{N}}$ depends continuously on $(t_{i,j})_{(i,j)\in\mathbb{N}^{2}}$.
\end{proof}
\begin{cor}
\label{cor: Pres and FullPres}The representations $\rho_{\mathrm{Pres}}$
and $\rho_{\mathrm{FullPres}}$ are equivalent. 
\end{cor}

\begin{proof}
The translation $\rho_{\mathrm{FullPres}}\le\rho_{\mathrm{Pres}}$
is immediate. And $\rho_{\mathrm{Pres}}\le\rho_{\mathrm{FullPres}}$
follows from the proof given above of $\rho_{\mathbb{S}\mathrm{WP}}\le\rho_{\mathrm{Pres}}$:
the constructed translation is easily seen to give a $\rho_{\mathrm{FullPres}}$-name
of the given marked group. 
\end{proof}

\subsection{The identity from $\mathcal{G}_{\mathrm{Pres}}$ to $\mathcal{G}_{\mathrm{WP}}$}

Here we study the complexity of $\mathrm{Id}:\mathcal{G}_{\mathrm{Pres}}\cap\mathcal{C}\rightarrow\mathcal{G}_{\mathrm{WP}}$,
i.e. of the identity function on a set $\mathcal{C}$ of groups, seen
as going from the set of marked groups equipped with the Scott topology
and the presentation representation to the space of marked groups
equipped with the word problem representation. 
\begin{lem}
\label{lem: Pres to WP is Lim }We have $\mathrm{Id}:\mathcal{G}_{\mathrm{Pres}}\rightarrow\mathcal{G}_{\mathrm{WP}}\equiv_{W}\mathrm{Lim}$. 
\end{lem}

\begin{proof}
The fact that $\mathrm{Id}:\mathcal{G}_{\mathrm{Pres}}\rightarrow\mathcal{G}_{\mathrm{WP}}\le_{W}\mathrm{Lim}$
follows directly from the corresponding well known result for the
Cantor space \cite{Brattka2021a}: 
\[
\mathrm{Id}:\mathbb{S}^{\mathbb{N}}\rightarrow\{0,1\}^{\mathbb{N}}\equiv_{W}\widehat{\mathrm{LPO}}\equiv_{W}\mathrm{Lim}.
\]
Indeed, for each $k\in\mathbb{N}$, $\mathcal{G}_{\mathrm{Pres}}^{k}$
is a subset of $\mathbb{S}^{\mathbb{N}}$ equipped with the induced
topology, and $\mathcal{G}_{\mathrm{WP}}^{k}$ is a subset of $\{0,1\}^{\mathbb{N}}$
also equipped with the induced topology. And thus $\mathrm{Id}:\mathcal{G}_{\mathrm{Pres}}\rightarrow\mathcal{G}_{\mathrm{WP}}$
is the restriction of a problem below $\mathrm{Lim}$, it is again
below $\mathrm{Lim}$. 

To prove that $\mathrm{Id}:\mathcal{G}_{\mathrm{Pres}}\rightarrow\mathcal{G}_{\mathrm{WP}}\ge_{W}\mathrm{Lim}$,
we use an infinite set of independent relations. For $A\subseteq\mathbb{N}$,
consider the marked group 
\[
(G_{A},S)=\langle a,b,c,d\,\mid\,a^{-i}ba^{i}=c^{-i}dc^{i},\,i\in A\rangle.
\]
 Using the fact that $G_{A}$ is defined as an amalgamated product,
it is easy to see that $\Psi:A\mapsto(G_{A},S)$ is injective, and
that: 
\begin{enumerate}
\item The map $\Psi_{1}:\mathbb{S}^{\mathbb{N}}\rightarrow\mathcal{G}_{\mathrm{Pres}}$
is a homeomorphism onto its image, 
\item The map $\Psi_{2}:\{0,1\}^{\mathbb{N}}\rightarrow\mathcal{G}_{\mathrm{WP}}$
is a homeomorphism onto its image. 
\end{enumerate}
It follows from this that $\mathrm{Id}:\mathbb{S}^{\mathbb{N}}\rightarrow\{0,1\}^{\mathbb{N}}$
can be written as $\Psi_{2}^{-1}\circ(\mathrm{Id}:\mathcal{G}_{\mathrm{Pres}}\rightarrow\mathcal{G}_{\mathrm{WP}})\circ\Psi_{1}$,
this gives the desired reduction for $\mathrm{Lim}\le_{W}\mathrm{Id}:\mathcal{G}_{\mathrm{Pres}}\rightarrow\mathcal{G}_{\mathrm{WP}}$.
\end{proof}
The following lemma is proved exactly as above, we are simply explicitly
asking for the existence of a function $\Psi$ such as appeared in
the above proof. 
\begin{lem}
\label{lem: Pres to WP is Lim -1}Let $\mathcal{C}$ be a set of finitely
generated groups. Suppose that there is an injective function $\Psi:\{0,1\}^{\mathbb{N}}\rightarrow\mathcal{G}\cap\mathcal{C}$,
and that 
\begin{enumerate}
\item The function $\Psi:\mathbb{S}^{\mathbb{N}}\rightarrow\mathcal{G}_{\mathrm{Pres}}\cap\mathcal{C}$
is continuous, 
\item The function $\Psi^{-1}:\mathcal{G}_{\mathrm{WP}}\cap\mathrm{Im}(\Psi)\rightarrow\{0,1\}^{\mathbb{N}}$
is also continuous. 
\end{enumerate}
Then $\mathrm{Id}:\mathcal{G}_{\mathrm{Pres}}\cap\mathcal{C}\rightarrow\mathcal{G}_{\mathrm{WP}}\equiv_{W}\mathrm{Lim}$. 
\end{lem}

Note that the existence of a function $\Psi$ as above will always
rely on the existence of an \emph{infinitely independently presented
group relative to $\mathcal{C}$}. We introduce this notion here.
A classical notion is: 
\begin{defn}
[\cite{Bieri2014}]\label{def:INIP}Let $G$ be a group and $\{N_{i}\,\mid\,i\in I\}$
a family of normal subgroups of $G$. This family is called \emph{independent
}if no subgroup is contained in the subgroup generated by the remaining
groups: 
\[
\forall i_{0}\in I,N_{i_{0}}\not\subseteq\langle N_{i},\,i\neq i_{0}\rangle.
\]
We say that a finitely generated group $G$ is \emph{infinitely independently
presentable} (or \emph{INIP}) if there is a morphism $\pi:\mathbb{F}_{k}\twoheadrightarrow G$
from a free group of finite rank onto $G$ whose kernel is generated
by an infinite and independent family of normal subgroups. 
\end{defn}

We consider a relative version of the above. 
\begin{defn}
\label{def:INIP-1}Let $\mathcal{C}$ be a set of finitely generated
groups. Let $G$ be a group and $\{N_{i}\,\mid\,i\in I\}$ a family
of normal subgroups of $G$. This family is called \emph{independent
relative to $\mathcal{C}$ }if for every $i_{0}$ in $I$, there is
a quotient $\psi_{i_{0}}:G\rightarrow H$ onto a group $H\in\mathcal{C}$
such that 
\[
\psi_{i_{0}}(N_{i_{0}})\not\subseteq\langle\psi_{i_{0}}(N_{i}),\,i\neq i_{0}\rangle.
\]

We say that a finitely generated group $G$ is \emph{infinitely independently
presentable relative to $\mathcal{C}$} (or \emph{INIP relative to
$\mathcal{C}$}) if there is a morphism $\pi:\mathbb{F}_{k}\twoheadrightarrow G$
from a free group of finite rank onto $G$ whose kernel is generated
by an infinite family of normal subgroups which is independent relative
to $\mathcal{C}$. 
\end{defn}

\begin{prop}
Let $\mathcal{C}$ be a set of finitely generated groups. Suppose
that there exists a map $\Psi:\mathcal{P}(\mathbb{N})\rightarrow\mathcal{G}\cap\mathcal{C}$
that satisfies the conditions of Lemma \ref{lem: Pres to WP is Lim -1}.
Then $\Psi(\emptyset)$ is infinitely independently presentable relative
to $\mathcal{C}$. 
\end{prop}

\begin{proof}
One needs to remark that, because Scott open sets are upward closed,
the map $\Psi$ must be order preserving: if $A\subseteq B$, then
there is a marked group morphism $\Psi(A)\rightarrow\Psi(B)$. Consider
the morphism $\pi:\mathbb{F}_{k}\rightarrow\Psi(\emptyset)$ that
defines the marking of $\Psi(\emptyset)$. The family $\{N_{i}\,\mid\,i\in\mathbb{N}\}$
given by $N_{i}=\pi^{-1}(\ker(\Psi(\emptyset)\rightarrow\Psi(\{i\})))$
is the desired set of normal subgroups independent relative to $\mathcal{C}$. 
\end{proof}
\begin{prop}
\label{prop:Direct product stable =00003D> id finitely //}Let $\mathcal{C}$
be a set of finitely generated groups stable by finite direct products.
The problem $\mathrm{Id}:\mathcal{G}_{\mathrm{Pres}}\cap\mathcal{C}\rightarrow\mathcal{G}_{\mathrm{WP}}$
if finitely parallelizable. 
\end{prop}

\begin{proof}
Indeed, having a $\rho_{\mathrm{WP}}$ name for $(G,S)$ and one for
$(H,S')$ is equivalent to having one for $(G\times H,S\cup S')$.
\end{proof}
\begin{prop}
\label{prop:Scott=00003DAlexandrov }Let $\mathcal{C}$ be a set of
finitely generated groups, such that every group in $\mathcal{C}$
admits a finite presentation relative to $\mathcal{C}$. Then the
Alexandrov and Scott topologies coincide on the set of marked groups
in $\mathcal{C}$. 
\end{prop}

\begin{proof}
The Alexandrov topology is always finer than the Scott topology. Thus
all we have to show is that any open in the Alexandrov topology is
also Scott open. But a basis of the Alexandrov topology is given by
sets of the form ``all marked quotients of $(G,S)$''. But if $(G,S)$
admits a finite presentation $\langle S\mid R\rangle$ relative to
$\mathcal{C}$, this set is identical to ``all marked groups that
satisfy the relations of $R$'', and this is a Scott open set.
\end{proof}
\begin{prop}
\label{thm: Pres-> WP is below ordinal length }Suppose that $\mathcal{C}$
is equationally noetherian. Let $k\in\mathbb{N}$. Let $\alpha$ be
the rank of the lattice of residually-$\mathcal{C}$ $k$-marked groups.
Then 
\[
\mathrm{Id}:\mathcal{G}_{\mathrm{Pres}}^{k}\cap\mathcal{RC}\rightarrow\mathcal{G}_{\mathrm{WP}}^{k}\le_{W}\mathrm{Lim}_{\searrow\alpha}.
\]
\end{prop}

\begin{proof}
This follows from directly from Proposition \ref{prop: ordinal-rank order is CB rank alexandrov top},
Proposition \ref{prop:Scott=00003DAlexandrov } and Proposition \ref{prop: CB rank upper bound discretization }. 
\end{proof}
We now show that when the rank of the lattice of residually-$\mathcal{C}$
$k$-marked groups is below $\omega^{\omega}$, the previous bound
is actually sharp. 
\begin{thm}
\label{thm: Pres-> WP is equal to ordinal length if alpha is small}Suppose
that $\mathcal{C}$ is equationally noetherian. Let $k\in\mathbb{N}$.
Let $\alpha$ be the rank of the lattice of residually-$\mathcal{C}$
$k$-marked groups. Suppose that $\alpha<\omega^{\omega}$. Then 
\[
\mathrm{Id}:\mathcal{G}_{\mathrm{Pres}}^{k}\cap\mathcal{RC}\rightarrow\mathcal{G}_{\mathrm{WP}}^{k}\equiv_{W}\mathrm{Lim}_{\searrow\alpha}.
\]
\end{thm}

\begin{proof}
For conciseness, we denote the problem $\mathrm{Id}:\mathcal{G}_{\mathrm{Pres}}^{k}\cap\mathcal{RC}\rightarrow\mathcal{G}_{\mathrm{WP}}^{k}$
by $P$. 

Consider a multi-injection $e$ of $\alpha$ into the lattice of residually-$\mathcal{C}$
$k$-marked groups, as given by Proposition \ref{prop: ordinal rank and multi-injection}. 

We will in fact not only use this multi-injection $e$, but also its
continuous realizer given in the proof of Proposition \ref{prop: ordinal rank and multi-injection}:
fix a well order $\succeq$ on $\mathcal{G}_{\mathrm{Pres}}^{k}\cap\mathcal{RC}$,
the realizer of $e$ is then defined by mapping a weakly descending
sequence $(\beta_{n})\in\alpha^{\mathbb{N}}$ to the sequence $((G_{n},S_{n}))\in\mathcal{RC}^{\mathbb{N}}$
given by 
\[
(G_{n+1},S_{n+1})=\inf_{\succeq}\left\{ (G,S)\mid\mathrm{rank}_{\mathcal{G}_{\mathrm{Pres}}^{k}}((G,S))=\beta_{n+1}\,\&\,(G_{n},S_{n})\rightarrow(G,S)\right\} .
\]

This descending sequence of marked groups can be seen as a single
group presentation by putting together all the relations of the marked
groups in $\{(G_{n},S_{n})\}$.

Denote by $E$ this continuous realizer of $e$. 

We will prove that for all natural number $n$: 
\begin{enumerate}
\item If $P\ge_{W}\mathrm{Lim}_{\searrow\omega^{n}}$ and $\alpha>\omega^{n+1}$,
then $P\ge\mathrm{Lim}_{\searrow\omega^{n+1}}$.
\item If $P\ge_{W}\mathrm{Lim}_{\searrow\omega^{n}}$ and $\omega^{n+1}\ge\alpha$,
then $P\ge\mathrm{Lim}_{\searrow\alpha}$.
\end{enumerate}
Note that $P\ge_{W}\mathrm{Lim}_{\searrow\omega}$ is automatic because
$P$ is finitely parallelizable (by Proposition \ref{prop:Direct product stable =00003D> id finitely //},
noting that $\mathcal{RC}$ is stable by direct products) and discontinuous,
and thus (1) and (2) together imply that $P\ge\mathrm{Lim}_{\searrow\alpha}$. 

We first prove (1). Suppose that $P\ge_{W}\mathrm{Lim}_{\searrow\omega^{n}}$.
Note that $\omega^{n+1}=\omega\omega^{n}$, and the elements of $\omega^{n+1}$
can be seen as pairs $(u,v)$ with $u\in\omega$ and $v\in\omega^{n}$,
the order being lexicographic with priority given to the right hand
side. 

We thus define maps $\pi_{1}:\omega^{n+1}\rightarrow\omega,\,\delta=(u,v)\mapsto u$
and $\pi_{2}:\omega^{n+1}\rightarrow\omega^{n},\,\delta=(u,v)\mapsto v$.
The map $\pi_{2}$ is order preserving. 

Applying $E$ to a name $u$ of some $\delta\in\omega^{n+1}$ yields
a group presentation $E(u)$ of a marked group with ordinal rank $\delta$. 

The problem $P$ can be used to convert this $\rho_{\mathrm{Pres}}$-name
into a $\rho_{\mathrm{WP}}$-name for $E(u)$. 

Also, by hypothesis, the problem $P$ can be used to obtain a name
of $\pi_{2}(\delta)$ for the discrete topology. 

And because $P$ is finitely parallelizable, one application of $P$
is sufficient to get both results. 

We now prove that $\delta$ can be recovered from this data. 

Since we have $\pi_{2}(\delta)$, it suffices to compute $\pi_{1}(\delta)$. 

An upper bound to $\pi_{1}(\delta)$ can be computed from $\pi_{2}(\delta)$,
simply by waiting for the weakly descending sequence that describes
$\delta$ to reach a point of the form $(v,\pi_{2}(\delta))$, at
which point we know that $\pi_{1}(\delta)\in\{0,...,v\}$. 

Once this is known, there are finitely many possibilities for what
the group $E(u)$ is. Indeed, $E(u)$ is either the marked group with
ordinal rank $(v,\pi_{2}(\delta))$ that $E$ produces upon the beginning
of the name of $\delta$, or its first (with respect to the well order
$\succeq$) quotient with ordinal rank $(v-1,\pi_{2}(\delta))$, or
its first quotient with ordinal rank $(v-2,\pi_{2}(\delta))$, ...,
and so on until $(0,\pi_{2}(\delta))$. 

But having access to the word problem of this marked group provides
a sufficient amount of information to determine which alternative
holds, since the topology of the space of marked groups induces the
discrete topology on any finite set of marked groups. 

We now prove (2), the proof is almost identical to that of (1). 

Suppose that $P\ge_{W}\mathrm{Lim}_{\searrow\omega^{n}}$ and $\omega^{n+1}\ge\alpha$.
We again identify $\omega^{n+1}$ with $\omega\times\omega^{n}$ via
the projection $(\pi_{1},\pi_{2})$. Given $\delta\in\alpha$ and
a name $u$ of it, we can use $P$ to compute a $\rho_{\mathrm{WP}}$-name
of $E(u)$ and a name of $\pi_{2}(\delta)$ for the discrete topology.
We conclude as for (1). 
\end{proof}
Notice that we have also proved: 
\begin{thm}
\label{thm: Pres-> WP above omega^omega when alpha is }Suppose that
$\mathcal{C}$ is equationally noetherian. Let $k\in\mathbb{N}$.
Let $\alpha$ be the rank of the lattice of residually-$\mathcal{C}$
$k$-marked groups. Suppose that $\alpha\ge\omega^{\omega}$. Then
\[
\mathrm{Id}:\mathcal{G}_{\mathrm{Pres}}^{k}\cap\mathcal{RC}\rightarrow\mathcal{G}_{\mathrm{WP}}^{k}\ge_{W}\mathrm{Lim}_{\searrow\omega^{\omega}}.
\]
\end{thm}

\begin{proof}
Following the proof of Theorem \ref{thm: Pres-> WP is equal to ordinal length if alpha is small},
we see that situation (2) never arises, and thus by using (1) we get
that $\mathrm{Id}:\mathcal{G}_{\mathrm{Pres}}^{k}\cap\mathcal{RC}\rightarrow\mathcal{G}_{\mathrm{WP}}^{k}\ge_{W}\mathrm{Lim}_{\searrow\omega^{n}}$
for every $n$. We conclude by Lemma \ref{prop: limit (ordinal problem) =00003D problem (limit ordinal)}. 
\end{proof}
The following corollary is immediate. 
\begin{cor}
Suppose that $\mathcal{C}$ is equationally noetherian. Let $\alpha_{k}$
be the ordinal rank of the lattice of residually-$\mathcal{C}$ $k$-marked
groups, and let $\alpha=\sup(\alpha_{k})$. If, for each $k$, $\alpha_{k}<\omega^{\omega}$,
then 
\[
\mathrm{Id}:\mathcal{G}_{\mathrm{Pres}}\cap\mathcal{RC}\rightarrow\mathcal{G}_{\mathrm{WP}}\equiv_{W}\mathrm{Lim}_{\searrow\alpha}.
\]
\end{cor}

\begin{prop}
\label{thm: Pres-> WP for not eq noeth}Suppose that $\mathcal{C}$
is not equationally noetherian. Then 
\[
\mathrm{Id}:\mathcal{G}_{\mathrm{Pres}}^{k}\cap\mathcal{RC}\rightarrow\mathcal{G}_{\mathrm{WP}}^{k}\ge_{W}\mathrm{Lim}_{\mathbb{N}}.
\]
\end{prop}

\begin{proof}
By Lemma \ref{lem: Eq noeth characterization } there is a marked
residually $\mathcal{C}$ group which does not admit a finite presentation
as a residually $\mathcal{C}$ group. Let $\langle S\mid r_{1},r_{2},r_{3}...\rangle$
be a presentation for this group. We can suppose w.l.o.g\. that each
marked quotient $\mathrm{Res}_{\mathcal{C}}\langle S\mid r_{1},...,r_{n}\rangle\rightarrow\mathrm{Res}_{\mathcal{C}}\langle S\mid r_{1},...,r_{n},r_{n+1}\rangle$
is a strict quotient. 

For each $n\in\mathbb{N}$ fix an infinite presentation $\pi_{n}=\langle S\mid t_{1}^{n},t_{2}^{n},t_{3}^{n}...\rangle$
of $\mathrm{Res}_{\mathcal{C}}\langle S\mid r_{1},...,r_{n}\rangle$. 

Let $(u_{n})_{n\in\mathbb{N}}\in\mathbb{N}^{\mathbb{N}}$ be a converging
sequence. Define a presentation associated to $(u_{n})_{n\in\mathbb{N}}\in\mathbb{N}^{\mathbb{N}}$
as follows. Define a function $m_{c}$ by $m_{c}(n)=\#\{k\le n\mid u_{k}\ne u_{k+1}\}$.
Note that because $(u_{n})_{n\in\mathbb{N}}$ converges, $m_{c}(n)$
converges to a certain limit $l$. 

Define a presentation $\Pi=\langle S\mid v_{1},v_{2},v_{3},...\rangle$
by 
\[
v_{\langle p,q\rangle}=t_{q}^{m_{c}(p)}.
\]
It follows that $\Pi$ defines the same marked group as $\pi_{l}$
(with many redundant relators). Denote this marked group $(G_{\Pi},S)$.
Note that $(G_{\Pi},S)\in\mathcal{RC}$. 

Applying $\mathrm{Id}:\mathcal{G}_{\mathrm{Pres}}^{k}\cap\mathcal{RC}\rightarrow\mathcal{G}_{\mathrm{WP}}^{k}$
to $\Pi$ yields a solution to the word-problem in $(G_{\Pi},S)$.
From this solution, the least $n_{0}$ such that $r_{n_{0}}$ is not
a relation of $(G_{\Pi},S)$ can be computed. 

It follows that the value of $(u_{n})_{n\in\mathbb{N}}$ changes exactly
$n_{0}$ times. This information is sufficient to recover the limit
of $(u_{n})_{n\in\mathbb{N}}$. 
\end{proof}

\section{general results on the map $\mathrm{Res}_{\mathcal{C}}$}
\begin{prop}
\label{thm: Quasi Var iff ResC pres pres is C0}Let $\mathcal{C}$
be a set of finitely generated groups. The following are equivalent: 
\begin{enumerate}
\item The set $\mathcal{RC}$ is a quasivariety;
\item The function $\mathrm{Res}_{\mathcal{C}}:\mathcal{G}_{\mathrm{Pres}}\rightarrow\mathcal{G}_{\mathrm{Pres}}$
is continuous;
\item The function $\mathrm{Res}_{\mathcal{C}}:\mathcal{G}_{\mathrm{WP}}\rightarrow\mathcal{G}_{\mathrm{Pres}}$
is continuous.
\end{enumerate}
\end{prop}

\begin{proof}
$(1)\implies(2)$. Suppose that $\mathcal{RC}$ is a quasivariety.
Given a presentation $\pi=\langle s_{1},...,s_{k}\mid R\rangle$ for
a marked group $(G,S)$, a presentation of $\mathrm{Res}_{\mathcal{C}}((G,S))$
is obtained by adding as further relators all elements $w$ that appear
in a quasi-identity $\bigwedge r_{i}=1\implies w=1$ of $\mathcal{RC}$. 

By Corollary \ref{cor: Pres and FullPres}, there is a continuous
map $F$ which maps any presentation of $(G,S)$ to a full presentation
of $(G,S)$, i.e. one where all relations of $(G,S)$ appear as relators.
There is also a continuous map $H$ which maps a presentation of $(G,S)$
to a presentation in which the ``obvious consequences'' of the quasi-identities
of $\mathcal{RC}$ are added: for $\pi$ a presentation, $H(\pi)$
contains the relations of $\pi$ and all relations $w$ which appear
in the right hand side of a quasi-identity $\bigwedge_{1\le i\le n}r_{i}=1\implies w=1$
of $\mathcal{RC}$ for which $(r_{1},...,r_{n})$ are relations that
appear in $\pi$. The composition of $F$ and $H$ yields the desired
continuous realizer of $\mathrm{Res}_{\mathcal{C}}:\mathcal{G}_{\mathrm{Pres}}\rightarrow\mathcal{G}_{\mathrm{Pres}}$.

$(2)\implies(3)$ is immediate since the space of marked groups topology
is finer than the Scott topology. 

$(3)\implies(1)$. Suppose that $\mathcal{RC}$ is not a quasivariety.
Then there exists $k\in\mathbb{N}$, an infinite set of relations
$(r_{1},r_{2},...)\in\mathbb{F}_{k}^{\mathbb{N}}$, and a relation
$w\in\mathbb{F}_{k}$ such that any marked group in $\mathcal{RC}$
which satisfies all the relations $(r_{1},r_{2},...)$ also satisfies
$w=1$, and yet for every finite subset $A$ of $(r_{1},r_{2},...)$
there is a group in $\mathcal{RC}$ which satisfies $A$ and $w\neq1$.
Consider the sequence $((G_{n},S_{n}))_{n\in\mathbb{N}}$ of finitely
presented groups given by $(G_{n},S_{n})=\langle s_{1},...,s_{k}\mid r_{1},...,r_{n}\rangle$.
The sequence $((G_{n},S_{n}))_{n\in\mathbb{N}}$ being increasing,
it converges towards a marked group $(G,S)$. And yet, by construction,
$w$ is a relation in $\mathrm{Res}_{\mathcal{C}}((G,S))$, but it
is never a relation of $\mathrm{Res}_{\mathcal{C}}((G_{n},S_{n}))$.
Thus $(\mathrm{Res}_{\mathcal{C}}((G_{n},S_{n})))_{n\in\mathbb{N}}$
does not converge to $\mathrm{Res}_{\mathcal{C}}((G,S))$.
\end{proof}
\begin{prop}
\label{prop: Finite parallelization}Let $\mathcal{C}$ be a set of
finitely generated groups. The map $\mathrm{Res}_{\mathcal{C}}$ is
finitely parallelizable. 
\end{prop}

\begin{proof}
It suffices to show that $\mathrm{Res}_{\mathcal{C}}\equiv_{W}\mathrm{Res}_{\mathcal{C}}\times\mathrm{Res}_{\mathcal{C}}$,
which is obvious: computing the residually $\mathcal{C}$ image of
$(G,S)$ and that of $(H,S')$ is equivalent to computing that of
$(G\times H,S\cup S')$.
\end{proof}
\begin{prop}
[Arbitrarily complicated $\mathrm{Res}_{\mathcal{C}}$] \label{prop:Arbitrarily-complicated Res C}There
are $2^{2^{\aleph_{0}}}$ problems of the form $\mathrm{Res}_{\mathcal{C}}$,
and thus some are not Borel measurable. 
\end{prop}

\begin{proof}
Consider a set $\mathcal{S}$ of finitely generated simple groups.
Consider $\mathcal{C}_{\mathcal{S}}=\{G\mid\text{No group of \ensuremath{\mathcal{S}} embeds in \ensuremath{G}}\}$. 

We consider the map $\mathrm{Res}_{\mathcal{\mathcal{C}_{\mathcal{S}}}}:\mathcal{G}_{\mathrm{WP}}\rightarrow\mathcal{G}_{\mathrm{WP}}$. 

The restriction of $\mathrm{Res}_{\mathcal{\mathcal{C}_{\mathcal{S}}}}$
to the (Polish) space of marked simple groups is Weihrauch equivalent
to the characteristic function of $\mathcal{S}$ (given a simple group,
is it in $\mathcal{S}$?), since non-trivial morphisms from simple
groups are injective. 

Since there are $2^{\aleph_{0}}$ finitely generated simple groups,
there are $2^{2^{\aleph_{0}}}$ sets of simple groups, and thus $2^{2^{\aleph_{0}}}$
different problems $\mathrm{Res}_{\mathcal{\mathcal{C}_{\mathcal{S}}}}$.
Notice that Weihrauch equivalence classes of problems defined on the
space of marked groups have cardinality at most the continuum, and
thus these $2^{2^{\aleph_{0}}}$ problems yield $2^{2^{\aleph_{0}}}$
Weihrauch degrees. 
\end{proof}
\begin{prop}
\label{prop:Countable}Let $\mathcal{C}$ be a countable set of finitely
generated groups. Then $\mathrm{Res}_{\mathcal{C}}:\mathcal{G}_{\mathrm{WP}}\rightarrow\mathcal{G}_{\mathrm{WP}}\le_{W}\mathrm{Lim}'$. 
\end{prop}

\begin{proof}
By Theorem \ref{thm: Lim^(n) is measurable complete} it suffices
to show that $\mathrm{Res}_{\mathcal{C}}$ is $\Sigma_{3}^{0}$-measurable. 

Let $\mathcal{C}=\{(G_{1},S_{1}),(G_{2},S_{2}),...\}$ be an enumeration
of all $k$-marked groups of $\mathcal{C}$. 

Consider a basic clopen set $\Omega$ of the space of $k$-marked
groups defined by two tuples of words $(w_{1},...,w_{n})$ and $(r_{1},...,r_{n'})$:
the set of $k$-marked groups which satisfy the relations $(r_{1},...,r_{n'})$
and in which $(w_{1},...,w_{n})$ are not relations.

Decompose $\Omega$ as a finite intersection: groups in which $w\ne1$,
for each $w\in\{w_{1},...,w_{n}\}$, and in which $r=1$, for each
$r\in\{r_{1},...,r_{n'}\}$. 

Notice that for each $n\in\mathbb{N}$, the set of marked groups that
map onto $(G_{n},S_{n})$, $\{\Gamma\in\mathcal{G}_{k}\mid\Gamma\rightarrow(G_{n},S_{n})\}$,
is a closed set, and $\{\Gamma\in\mathcal{G}_{k}\mid\Gamma\nrightarrow(G_{n},S_{n})\}$
is open. 

And we have that $w\ne1$ in $\mathrm{Res}_{\mathcal{C}}((H,S))$
if and only if 
\[
(H,S)\in\bigcup_{\{n\mid w\ne1\text{ in }G_{n}\}}\{\Gamma\in\mathcal{G}_{k}\mid\Gamma\rightarrow G_{n}\}
\]
the right hand side condition being a $F_{\sigma}$.

And $r=1$ in $\mathrm{Res}_{\mathcal{C}}((H,S))$ if and only if
\[
(H,S)\in\bigcap_{\{n\mid r\ne1\text{ in }G_{n}\}}\{\Gamma\in\mathcal{G}_{k}\mid\Gamma\nrightarrow G_{n}\}
\]
the right hand side condition being a $G_{\delta}$.

We thus find that the preimage by $\mathrm{Res}_{\mathcal{C}}$ of
a basic open set is a finite intersection of $F_{\sigma}$ and $G_{\delta}$
sets, it is thus in $\Sigma_{3}^{0}$. 

And thus the map $\mathrm{Res}_{\mathcal{C}}$ is indeed $\Sigma_{3}^{0}$-measurable.
\end{proof}

\begin{defn}
We say that that \emph{$\mathcal{RC}$ is non-trivial }if it is neither
the set $\mathcal{G}$ of all marked groups nor reduced to the trivial
group. 
\end{defn}

\begin{lem}
[Small cancellation lemma]\label{lem:Small-cancellation-lemma}Let
$\mathcal{C}$ be a set of finitely generated groups. Suppose that
$\mathcal{RC}$ is non-trivial. Then:
\[
\mathrm{Res}_{\mathcal{C}}:\mathcal{G}_{\mathrm{WP}}\rightarrow\mathcal{G}_{\mathrm{WP}}\equiv_{W}\mathrm{Res}_{\mathcal{C}}:\mathcal{G}_{\mathrm{Pres}}\rightarrow\mathcal{G}_{\mathrm{WP}}.
\]
\end{lem}

\begin{proof}
Fix a marked group $(H,T)$ which is not residually $\mathcal{C}$,
and an element $w\in H\setminus\{1\}$ which is killed by all morphisms
going from $H$ to groups of $\mathcal{C}$. 

Let $(G,S)$ be a marked group given by a presentation $\langle S\mid r_{1},r_{2},...\rangle$. 

Up to replacing $G$ by $G*\mathbb{Z}$, we can suppose that $S$
contains at least two elements $a$ and $b$. Let $\mathbb{F}_{S}$
be the free group on $S$. 

We define a group $\Gamma$ as a quotient of the free product $\mathbb{F}_{S}*H$,
quotiented by relations $w_{i}=1$, $i\ge5$, where $w_{i}$ is given
by 
\[
w_{i}=r_{i-4}wa^{i^{2}}wa^{-i^{2}}b^{i^{2}}wb^{-i^{2}}a^{i^{2}+1}wa^{-i^{2}-1}...a^{i^{2}+2i}wa^{-i^{2}-2i}b^{i^{2}+2i}wb^{-i^{2}-2i}w.
\]
We equip $\Gamma$ with the generating set $(S,T)$.

We now show: 
\begin{enumerate}
\item \label{enu:Pres -> WP C0}The word-problem of $(\Gamma,(S,T))$ can
be continuously constructed from the presentation of $(G,S)$, 
\item \label{enu:WP res -> WP C0}The word-problem of $(\mathrm{Res}_{\mathcal{C}}(G),S)$
can be continuously constructed from the word-problem of $(\mathrm{Res}_{\mathcal{C}}(\Gamma),(S,T))$. 
\end{enumerate}
Those two facts together immediately imply that we have a Weihrauch
reduction $\mathrm{Res}_{\mathcal{C}}:\mathcal{G}_{\mathrm{WP}}\rightarrow\mathcal{G}_{\mathrm{WP}}\ge_{W}\mathrm{Res}_{\mathcal{C}}:\mathcal{G}_{\mathrm{Pres}}\rightarrow\mathcal{G}_{\mathrm{WP}}$,
and, the other reduction being trivial, the lemma follows. 

We first prove (\ref{enu:Pres -> WP C0}). 

This relies on the small cancellation condition $C(1/6)$ in free
products. We briefly recall some basic facts from \cite[Section 9, Chapter V]{Lyndon1977}. 

For an element $g$ of a free product $A*B$, its \emph{length} $\left|g\right|$
is the minimal $n$ such that $g$ can be written as a product $x_{1}...x_{n}$,
where consecutive elements $x_{i}$, $x_{i+1}$ do not belong to the
same factor. 

A product $g=uv$ is in \emph{reduced form}\textbf{ }if $\left|g\right|=\left|u\right|+\left|v\right|$,
and in \emph{semi-reduced form} if $\left|g\right|\ge\left|u\right|+\left|v\right|-1$
(i.e. no cancellation occurs in the product, but $u$ could end by
an element of the same factor as the starting element of $v$). 

An element $b$ is a \emph{piece }in a set $R\subseteq A*B$ if there
are distinct elements $r_{1}$ and $r_{2}$ in $R$ with $r_{1}=bc_{1}$
and $r_{2}=bc_{2}$, where both products are in semi-reduced form. 

A set $R$ satisfies \emph{condition $C'(1/6)$ for free products}
if for any $r$ in $R$, if $r$ can be expressed as a semi-reduced
product $r=bc$, where $b$ is a piece for $R$, then $\left|b\right|<\left|r\right|/6$. 

A set $R\subseteq A*B$ is is called \emph{symmetrized} if it is stable
by conjugation-reduction.

We will use \cite[Theorem 9.3]{Lyndon1977}: let $R$ be a symmetrized
subset of a free product $A*B$ satisfying the $C'(1/6)$ condition
for free products. Then any non-identity element $x$ of the normal
closure of $R$ in $A*B$ has a unique expression 
\[
x=usv,
\]
where the product is in reduced form, and there is some $r$ in $R$
such that $r=st$ in reduced form and $\left|s\right|>\left|r\right|/2$.
In words: every relation contains at least half of a relation of one
of the elements in $R$. 

Below, we show that the given presentation of $\Gamma$ satisfies
this condition $C'(1/6)$. It follows that the relations of $\Gamma$
of length less than a given $i$ do not depend on the relations $r_{j}$
of $G$ for $j\ge i$, which ensures continuity of the map mapping
the presentation of $G$ to the word-problem of $\Gamma$, and thus
(\ref{enu:Pres -> WP C0}) holds. 

Thus all that is left to establish (\ref{enu:Pres -> WP C0}) is to
prove that the set of relations $\{w_{i}\}$ together with their cyclically
reduced conjugates satisfy the condition $C'(1/6)$. It is clear that
the words $w_{i}$ are themselves cyclically reduced. 

Each word $w_{i}$ is a product of elements of $\mathbb{F}_{S}$ together
with the element $w$ which belongs to $H$. Notice that if a word
$wuw$ appears both as a subword of $w_{i}$ and of $w_{j}$, for
$j\ne i$, it must be that $u$ is one of the relations of the presentation
$\langle S\mid r_{1},r_{2},...\rangle$ of $G$ (since the other elements
of $\mathbb{F}_{S}$ that appear in the relations $w_{i}$ were chosen
distinct). And thus a piece has length at most $7$: it must have
the form 
\[
xwywzwt,
\]
where $(x,y,z,t)\in\mathbb{F}_{S}$, and $y$ and $z$ are both relations
$r_{i}$. 

Notice that the relation $w_{i}$ has length $8i+6$, and thus the
shortest relation of $\Gamma$, $w_{5}$, has length $46$, and $7/46<1/6$. 

And thus the small cancellation condition holds. 

We now prove (\ref{enu:WP res -> WP C0}). 

In the group $\mathrm{Res}_{\mathcal{C}}(\Gamma)$, the element $w$
is trivial. Notice that, thanks to $w=1$, each relation $w_{i}$
simplifies down to $r_{i}=1$. And thus, in the group $\mathrm{Res}_{\mathcal{C}}(\Gamma)$,
the elements of $S$ satisfy the relations of $G$. It follows that
they generate exactly the group $\mathrm{Res}_{\mathcal{C}}(G)$. 

And thus the word problem of $\mathrm{Res}_{\mathcal{C}}(G)$ can
be continuously constructed from the word problem of $\mathrm{Res}_{\mathcal{C}}(\Gamma)$. 
\end{proof}
The above lemma has several useful corollaries. 
\begin{cor}
\label{cor: Res C above Identity pres to wp }Let $\mathcal{C}$ be
a set of finitely generated groups. Suppose that $\mathcal{RC}$ is
non-trivial. Then:
\[
\mathrm{Res}_{\mathcal{C}}:\mathcal{G}_{\mathrm{WP}}\rightarrow\mathcal{G}_{\mathrm{WP}}\ge_{W}\mathrm{Id}:\mathcal{G}_{\mathrm{Pres}}\cap\mathcal{RC}\rightarrow\mathcal{G}_{\mathrm{WP}}.
\]
\end{cor}

\begin{proof}
We have 
\begin{align*}
\mathrm{Res}_{\mathcal{C}}:\mathcal{G}_{\mathrm{WP}}\rightarrow\mathcal{G}_{\mathrm{WP}}\equiv_{W} & \mathrm{Res}_{\mathcal{C}}:\mathcal{G}_{\mathrm{Pres}}\rightarrow\mathcal{G}_{\mathrm{WP}}\\
\ge_{W} & \mathrm{Res}_{\mathcal{C}}:\mathcal{G}_{\mathrm{Pres}}\cap\mathcal{RC}\rightarrow\mathcal{G}_{\mathrm{WP}}\\
\equiv_{W} & \mathrm{Id}:\mathcal{G}_{\mathrm{Pres}}\cap\mathcal{RC}\rightarrow\mathcal{G}_{\mathrm{WP}},
\end{align*}
where the first equivalence is given by Lemma \ref{lem:Small-cancellation-lemma},
and the reduction $\ge_{W}$ follows from the fact that the restriction
of a problem is always below this problem. 
\end{proof}
As a corollary to the above, we get that $\mathrm{Res}_{\mathcal{C}}$
is always discontinuous.
\begin{cor}
\label{prop: Minimum LPO*}Let $\mathcal{C}$ be a set of finitely
generated groups. Suppose that $\mathcal{RC}$ is non-trivial. Then
$\mathrm{Res}_{\mathcal{C}}$ is discontinuous, and thus $\mathrm{Res}_{\mathcal{C}}\ge_{W}\mathrm{LPO}^{*}\equiv_{W}\mathrm{Lim}_{\searrow\omega}$. 
\end{cor}

\begin{proof}
By Corollary \ref{cor: Res C above Identity pres to wp }, it suffices
to show that $\mathrm{Id}:\mathcal{G}_{\mathrm{Pres}}\cap\mathcal{RC}\rightarrow\mathcal{G}_{\mathrm{WP}}$
is discontinuous, which amounts to showing that the topology of the
space of marked groups is strictly finer than the Scott topology on
$\mathcal{RC}$. But $\mathcal{G}_{\mathrm{WP}}$ is metrizable and
thus $\mathrm{T}_{2}$, whereas the Scott topology is never $\mathrm{T}_{2}$
in any set that contains a marked group $(G,S)$ and a strict quotient
of it, since the constant sequence corresponding to the strict quotient
converges to the group $(G,S)$.
\end{proof}
Another corollary to Lemma \ref{lem:Small-cancellation-lemma} is
the following. 
\begin{cor}
\label{cor:Quasi Var ResC equiv to Pres->WP}Let $\mathcal{C}$ be
a set of finitely generated groups. If $\mathcal{RC}$ is a quasivariety,
then
\[
\mathrm{Res}_{\mathcal{C}}\equiv_{W}\mathrm{Id}:\mathcal{G}_{\mathrm{Pres}}\cap\mathcal{RC}\rightarrow\mathcal{G}_{\mathrm{WP}}.
\]
\end{cor}

\begin{proof}
By Corollary \ref{cor: Res C above Identity pres to wp }, it suffices
to show that $\mathrm{Res}_{\mathcal{C}}\le_{W}\mathrm{Id}:\mathcal{G}_{\mathrm{Pres}}\cap\mathcal{RC}\rightarrow\mathcal{G}_{\mathrm{WP}}$.
But by Theorem \ref{thm: Quasi Var iff ResC pres pres is C0}, because
$\mathcal{RC}$ is a quasivariety, the map 
\[
\mathrm{Res}_{\mathcal{C}}:\mathcal{G}_{\mathrm{WP}}\rightarrow\mathcal{G}_{\mathrm{Pres}}
\]
is continuous. The result then follows from the fact that 
\[
\mathrm{Res}_{\mathcal{C}}:\mathcal{G}_{\mathrm{WP}}\rightarrow\mathcal{G}_{\mathrm{WP}}=(\mathrm{Res}_{\mathcal{C}}:\mathcal{G}_{\mathrm{WP}}\rightarrow\mathcal{G}_{\mathrm{Pres}})\circ(\mathrm{Id}:\mathcal{G}_{\mathrm{Pres}}\cap\mathcal{RC}\rightarrow\mathcal{G}_{\mathrm{WP}}).\qedhere
\]
\end{proof}
\begin{prop}
\label{prop:Not Qvar computes LPO'}Suppose that $\mathcal{RC}$ is
not a quasivariety. Then 
\[
\mathrm{Res}_{\mathcal{C}}\ge_{W}\mathrm{LPO}',
\]
and thus $\mathrm{Res}_{\mathcal{C}}$ is not $\Sigma_{2}^{0}$-measurable. 
\end{prop}

\begin{proof}
If $\mathcal{RC}$ is not a quasivariety, then there is a finite set
$S$, an infinite set of relations $R=\{r_{0},r_{1},r_{2},...\}$
of $\mathbb{F}_{S}$, and a relation $w\in\mathbb{F}_{S}$, such that
every group in $\mathcal{RC}$ which satisfies all the relations of
$R$ also satisfies $w$, and yet for each finite subset $R_{1}$
of $R$, there is a group in $\mathcal{RC}$ satisfying the relations
of $R_{1}$ and in which $w\ne1$.

We build a reduction to prove that $\mathrm{Res}_{\mathcal{C}}\ge_{W}\mathrm{LPO}'$.
Recall from Example \ref{exa:LPO' Exists infty} that the problem
$\mathrm{LPO}'$ can be seen as the problem $\exists^{\infty}$, which
maps a sequence of zeroes and ones to $1$ if and only if it contains
infinitely many ones. 

On input a sequence $(u_{n})_{n\in\mathbb{N}}$, we build a group
presentation $\langle S\mid t_{0},t_{1},...\rangle$, given by $t_{i}=r_{j}$,
where $j=\#\{n\le i\mid u_{n}=1\}$. 

The map $(u_{n})_{n\in\mathbb{N}}\mapsto\langle S\mid t_{0},t_{1},...\rangle$
is clearly a continuous map from $\{0,1\}^{\mathbb{N}}$ to $\mathcal{G}_{\mathrm{Pres}}$.

And by construction $w=1$ in $\mathrm{Res}_{\mathcal{C}}(\langle S\mid t_{0},t_{1},...\rangle)$
if and only if $\exists^{\infty}n,\,u_{n}=1$.
\end{proof}
\begin{thm}
\label{thm:QVAR IFF}Let $\mathcal{C}$ be a set of finitely generated
groups. Then $\mathcal{RC}$ is a quasivariety if and only if $\mathrm{Res}_{\mathcal{C}}\le_{W}\mathrm{Lim}$.
\end{thm}

\begin{proof}
By Corollary \ref{cor:Quasi Var ResC equiv to Pres->WP}, if $\mathcal{RC}$
is a quasivariety, then $\mathrm{Res}_{\mathcal{C}}\equiv_{W}\mathrm{Id}:\mathcal{G}_{\mathrm{Pres}}\cap\mathcal{RC}\rightarrow\mathcal{G}_{\mathrm{WP}}$.
By Lemma \ref{lem: Pres to WP is Lim }, $\mathrm{Id}:\mathcal{G}_{\mathrm{Pres}}\rightarrow\mathcal{G}_{\mathrm{WP}}\equiv_{W}\mathrm{Lim}$.
Thus $\mathrm{Res}_{\mathcal{C}}\le_{W}\mathrm{Lim}$ when $\mathcal{RC}$
is a quasivariety. 

And by Proposition \ref{prop:Not Qvar computes LPO'}, if $\mathcal{RC}$
is not a quasivariety, then $\mathrm{Res}_{\mathcal{C}}\ge_{W}\mathrm{LPO}'$,
and thus $\mathrm{Res}_{\mathcal{C}}\not\le_{W}\mathrm{Lim}$. 
\end{proof}
\begin{thm}
\label{thm:Eq Noeth IFF}Let $\mathcal{C}$ be a set of finitely generated
groups. The class $\mathcal{C}$ is equationally noetherian if and
only if $\mathrm{Res}_{\mathcal{C}}<_{W}\mathrm{Lim}_{\mathbb{N}}$. 
\end{thm}

\begin{proof}
Suppose that $\mathcal{C}$ is not equationally noetherian. By Corollary
\ref{cor: Res C above Identity pres to wp }, $\mathrm{Res}_{\mathcal{C}}\ge_{W}\mathrm{Id}:\mathcal{G}_{\mathrm{Pres}}\cap\mathcal{RC}\rightarrow\mathcal{G}_{\mathrm{WP}}$,
and by Proposition \ref{thm: Pres-> WP for not eq noeth} $\mathrm{Id}:\mathcal{G}_{\mathrm{Pres}}\cap\mathcal{RC}\rightarrow\mathcal{G}_{\mathrm{WP}}\ge_{W}\mathrm{Lim}_{\mathbb{N}}$. 

Suppose now that $\mathcal{C}$ is equationally noetherian. By Lemma
\ref{lem:Eq noeth -> Res C is Qvar} and Corollary \ref{cor:Quasi Var ResC equiv to Pres->WP}
we have $\mathrm{Res}_{\mathcal{C}}\equiv_{W}\mathrm{Id}:\mathcal{G}_{\mathrm{Pres}}\cap\mathcal{C}\rightarrow\mathcal{G}_{\mathrm{WP}}$,
and by Proposition \ref{thm: Pres-> WP is below ordinal length }
we have, for some countable ordinal $\alpha$, $\mathrm{Id}:\mathcal{G}_{\mathrm{Pres}}\cap\mathcal{C}\rightarrow\mathcal{G}_{\mathrm{WP}}\le_{W}\mathrm{Lim}_{\searrow\alpha}<_{W}\mathrm{Lim}_{\mathbb{N}}$. 
\end{proof}
The following immediate corollary provides a gap in terms of Weihrauch
complexity for problems of the form $\mathrm{Res}_{\mathcal{C}}$. 
\begin{cor}
Let $\mathcal{C}$ be a set of finitely generated groups. If $\mathrm{Res}_{\mathcal{C}}\le_{W}\mathrm{Lim}_{\mathbb{N}}$,
then there exists a countable ordinal $\alpha$ such that $\mathrm{Res}_{\mathcal{C}}<_{W}\mathrm{Lim}_{\searrow\alpha}$. 
\end{cor}

An example of a problem which is below $\mathrm{Lim}_{\mathbb{N}}$
but below no $\mathrm{Lim}_{\searrow\alpha}$ is the problem $\mathrm{Lim}_{\{0,1\}}$,
mapping a converging sequence in $\{0,1\}$ to its limit. 

\section{Below $\mathrm{Lim}_{\mathbb{N}}$: equationally noetherian sets}

Lemma \ref{lem:Eq noeth -> Res C is Qvar}, Corollary \ref{cor:Quasi Var ResC equiv to Pres->WP}
and Theorem \ref{thm: Pres-> WP is below ordinal length } together
yield: 
\begin{thm}
\label{thm: Eq noeth + rank lattice res C}Let $\mathcal{C}$ be a
set of finitely generated groups. Suppose that $\mathcal{C}$ is equationally
noetherian. Denote by $\alpha$ the ordinal rank of the lattice $\mathcal{RC}$.
If $\alpha\le\omega^{\omega}$, then we have the equality:
\[
\mathrm{Res}_{\mathcal{C}}\equiv_{W}\mathrm{Lim}_{\searrow\alpha}.
\]
Otherwise, we have the bounds:
\[
\mathrm{Lim}_{\searrow\omega^{\omega}}\le_{W}\mathrm{Res}_{\mathcal{C}}\le_{W}\mathrm{Lim}_{\searrow\alpha}.
\]
\end{thm}

Thus when $\mathcal{C}$ is equationally noetherian, computing the
ordinal rank of the lattice $\mathcal{RC}$ is crucial in order to
place the problem $\mathrm{Res}_{\mathcal{C}}$ in the Weihrauch lattice. 

This rank is always at least $\omega$, as can be seen by using a
sequence of morphisms of direct products $G\leftarrow G\times G\leftarrow G\times G\times G\leftarrow...$,
and using markings with increasingly many generators. Equality is
easily seen to be attained for finite sets of finite groups. 
\begin{prop}
Let $\mathcal{C}$ be a finite set of finite groups. Then $\mathrm{Res}_{\mathcal{C}}\equiv_{W}\mathrm{Lim}_{\searrow\omega}$. 
\end{prop}

\begin{proof}
There are only finitely many $k$-marked residually-$\mathcal{C}$
groups when $\mathcal{C}$ is a finite set of finite groups. Thus
the rank of the lattice of residually $\mathcal{C}$ groups is $\omega$,
we conclude by Theorem \ref{thm: Eq noeth + rank lattice res C}.
\end{proof}
Let  $\mathcal{N}_{k}$ be the set of finitely generated $k$-nilpotent
groups and $\mathcal{M}$ the set of finitely generated metabelian
groups. 
\begin{thm}
\label{thm: Classification Fee Nilp Metab}The following hold: 
\[
\forall k>0,\,\mathrm{Res}_{\mathcal{N}_{k}}\equiv_{W}\mathrm{Lim}_{\searrow\omega^{2}},
\]
\[
\mathrm{Res}_{\mathcal{M}}\ge_{W}\mathrm{Lim}_{\searrow\omega^{\omega}}.
\]
\end{thm}

In \cite{Cornulier2011}, Cornulier asked if $\omega^{\omega}$ is
the Cantor-Bendixson rank of the space of marked groups. We ask the
same question, but for the Scott topology instead:
\begin{problem}
Let $\mathcal{C}$ be an equationally noetherian set of finitely generated
groups. Is $\omega^{\omega}$ an upper bound to the ordinal rank of
the lattice of marked groups in $\mathcal{RC}$? 
\end{problem}

Another problem is the following:
\begin{problem}
Characterize the Weihrauch complexities of problem of the form $\mathrm{Res}_{\{H\}}$,
for $H$ any Gromov hyperbolic group. 
\end{problem}

\subsection{Residually free image }

In \cite{Louder2012}, Louder proves that the rank of the lattice
of marked limit groups is $\omega$. Can this result be extended to
measure the rank of the lattice of residually free groups?
\begin{problem}
What is the ordinal rank of the lattice of residually free groups?
\textcolor{orange}{}
\end{problem}

\subsection{Abelianization and $p$-nilpotent images }

The classification of $\mathrm{Res}_{\mathcal{N}_{p}}$ in the Weihrauch
degrees follows from the following proposition. 
\begin{prop}
Let $G$ be a finitely generated nilpotent group of Hirsch length
$l$. There is $n\in\mathbb{N}$ such that the ordinal rank of the
lattice of marked quotients of $G$ is $\omega l+n$. 
\end{prop}

\begin{proof}
By induction on $l$, the case $l=0$ being trivial. 

Suppose the result true up to length $l$. 

A nilpotent group of Hirsch length $l+1$ has finitely many quotients
that also have Hirsch length $l+1$. And thus the rank of the lattice
of marked quotient of $G$ can be written as $\alpha+n$, where $n\in\mathbb{N}$
and $\alpha$ is the rank of the lattice of marked quotients of a
group $H$ of Hirsch length $l+1$ which admits no proper quotient
of Hirsch length $l+1$. 

By definition, the rank of the lattice of marked quotient of $H$
is the supremum of the ranks of the lattices of its proper quotients
plus one, by induction hypothesis each of these has the form $\omega l+p$,
and thus their supremum is at most $\omega(l+1)$. This supremum is
in fact attained, as can easily be seen as follows. The center of
$H$ must contain a copy of $\mathbb{Z}$ (because an infinite finitely
generated nilpotent group has an infinite center). Denote by $a$
a generator of this copy of $\mathbb{Z}$. We get a sequence of marked
quotients 
\[
H/\langle a^{2}\rangle\leftarrow H/\langle a^{4}\rangle\leftarrow H/\langle a^{8}\rangle\leftarrow...
\]
which shows that the supremum of the ranks of strict quotients of
$H$ is indeed $\omega(l+1)$. 
\end{proof}
\begin{cor}
For any $p\in\mathbb{N}$, the ordinal rank of the lattice of marked
$p$-nilpotent groups is $\omega^{2}$. 
\end{cor}

\subsection{Metabelian image }

The computation of the ordinal length of quotients of metabelian groups
relies on the following result of Cornulier: 
\begin{thm}
[\cite{Cornulier2011}]The Cantor-Bendixson rank of the space of
marked metabelian groups is $\omega^{\omega}$. 
\end{thm}

By ``space of metabelian groups'' we mean the subset of the space
of marked groups consisting of metabelian groups, equipped with the
subset topology. 

Because the topology of the space of marked groups is finer than the
Scott topology, and because the Scott topology coincides with the
Alexandrov topology on the lattice of marked metabelian groups (Proposition
\ref{prop:Scott=00003DAlexandrov }), we get:
\begin{cor}
The ordinal rank of the lattice of marked metabelian groups is at
least $\omega^{\omega}$. 
\end{cor}

\section{Between $\mathrm{Lim}_{\mathbb{N}}$ and $\mathrm{Lim}$}
\begin{thm}
Let $\mathcal{C}$ be a set of finitely generated groups. Suppose
that $\mathcal{RC}$ is a quasivariety and that it contains a group
INIP relative to $\mathcal{RC}$. Then 
\[
\mathrm{Res}_{\mathcal{C}}:\mathcal{G}_{\mathrm{WP}}\rightarrow\mathcal{G}_{\mathrm{WP}}\equiv_{W}\mathrm{Lim}.
\]
\end{thm}

\begin{proof}
By Corollary \ref{cor:Quasi Var ResC equiv to Pres->WP} it suffices
to show that 
\[
\mathrm{Id}:\mathcal{G}_{\mathrm{Pres}}\cap\mathcal{RC}\rightarrow\mathcal{G}_{\mathrm{WP}}\equiv_{W}\mathrm{Lim}.
\]
By Lemma \ref{lem: Pres to WP is Lim -1} it suffices to build an
injective map $\Psi:\{0,1\}^{\mathbb{N}}\rightarrow\mathcal{G}\cap\mathcal{RC}$
such that 
\[
\Psi:\mathbb{S}^{\mathbb{N}}\rightarrow\mathcal{G}_{\mathrm{Pres}}
\]
 is continuous, and 
\[
\Psi^{-1}:\mathcal{G}_{\mathrm{WP}}\cap\mathrm{Im}(\Psi)\rightarrow\{0,1\}^{\mathbb{N}}
\]
is also continuous. Let $(G,S)$ be a INIP group relative to $\mathcal{RC}$.
Let $\{r_{i},i\in\mathbb{N}\}\in\mathbb{F}_{k}^{\mathbb{N}}$ be a
set of relations that are independent relative to $\mathcal{RC}$:
for every $i_{0}$ there is a quotient $H$ of $\mathbb{F}_{k}$ in
$\mathcal{RC}$ such that the normal subgroup of $H$ generated by
$\{r_{i},i\neq i_{0}\}$ does not contain $r_{i_{0}}$. 

We define $\Psi:\{0,1\}^{\mathbb{N}}\rightarrow\mathcal{G}\cap\mathcal{RC}$
by 
\[
\Psi((u_{i})_{i\in\mathbb{N}})=\mathrm{Res}_{\mathcal{C}}(\mathbb{F}_{k}/\langle\langle r_{i},u_{i}=1\rangle\rangle).
\]
The fact that the set $\{r_{i},i\in\mathbb{N}\}$ is independent relative
to $\mathcal{RC}$ guarantees that $\Psi$ is injective. 

The fact that $\Psi:\mathbb{S}^{\mathbb{N}}\rightarrow\mathcal{G}_{\mathrm{Pres}}$
is continuous follows from the fact that $\mathcal{RC}$ is a quasivariety:
$\Psi$ is the composition of $\mathrm{Res}_{\mathcal{C}}:\mathcal{G}_{\mathrm{Pres}}\rightarrow\mathcal{G}_{\mathrm{Pres}}$,
which is continuous because $\mathcal{RC}$ is a quasivariety (Proposition
\ref{thm: Quasi Var iff ResC pres pres is C0}), with $(u_{i})_{i\in\mathbb{N}}\mapsto\mathbb{F}_{k}/\langle\langle r_{i},i\in A\rangle\rangle$,
which is obviously a continuous function between $\mathbb{S}^{\mathbb{N}}$
and $\mathcal{G}_{\mathrm{Pres}}$. 

Is is easy to see that $\Psi^{-1}:\mathcal{G}_{\mathrm{WP}}\cap\mathrm{Im}(\Psi)\rightarrow\{0,1\}^{\mathbb{N}}$
is continuous. Indeed, the $\rho_{\mathrm{WP}}$-name of a marked
group $(G,S)$ is the sequences $(u_{i})_{i\in\mathbb{N}}\in\{0,1\}^{\mathbb{N}}$
given by $u_{i}=1$ iff $\theta_{k}(i)=1$ in $(G,S)$, where $\#S=k$.
Then, $\Psi^{-1}((G,S))$ is the sub-sequence $(v_{i})_{i\in\mathbb{N}}$
of $(u_{i})_{i\in\mathbb{N}}$ given by 
\[
v_{i}=u_{\theta_{k}^{-1}(r_{i})}.
\]
The extraction map $(u_{i})_{i\in\mathbb{N}}\mapsto(v_{i})_{i\in\mathbb{N}}$
is continuous. 
\end{proof}
\begin{cor}
\label{cor: Res LEF k-solvable torsion free }Let $\mathcal{C}$ be
any of: 
\begin{itemize}
\item the set of finitely generated $k$-solvable groups, for $k>2$, 
\item the set of finitely generated exponent $k$ groups, for $k$ a big
enough odd number, 
\item the set of finitely generated LEF groups, 
\item the set of finitely generated torsion free groups, 
\item the set of finitely generated left orderable groups. 
\end{itemize}
Then 
\[
\mathrm{Res}_{\mathcal{C}}:\mathcal{G}_{\mathrm{WP}}\rightarrow\mathcal{G}_{\mathrm{WP}}\equiv_{W}\mathrm{Lim}.
\]
\end{cor}

\begin{proof}
For each $\mathcal{C}$, it suffices to exhibit a group which is INIP
relative to the quasivariety. 

For finitely generated $k$-solvable groups, finitely generated torsion
free groups, we can use P. Hall's central-by-metabelian group whose
center is $\bigoplus_{i\in\mathbb{N}}\mathbb{Z}$ \cite{Hall1954}.
See Section \ref{subsec:Hall's-central-by-metabelian-gro}. 

This group also works for left orderability. Indeed, an extension
of a left-orderable group by a left-orderable group is left orderable
\cite[Problem 1.8]{Clay2016}. And thus any torsion free central quotient
of Hall's group remains left-orderable.

For LEF groups, we can use Dyson's doubles of lamplighter groups from
\cite{Dyson1974}, see Proposition \ref{prop:INIP for LEF groups}. 

For groups of finite, big enough, odd exponent, see \cite{Adian2018},
where groups of finite exponent whose center is infinite dimensional
are constructed. 
\end{proof}

\section{Functions equivalent to $\mathrm{Lim}'$}

Here we establish: 
\begin{thm}
\label{thm:Finite nilpotent FP}Let $\mathcal{C}$ be any of: 
\begin{itemize}
\item the set of finite groups, 
\item the set of finitely generated nilpotent groups, 
\item the set of finitely presentable groups. 
\end{itemize}
Then 
\[
\mathrm{Res}_{\mathcal{C}}:\mathcal{G}_{\mathrm{WP}}\rightarrow\mathcal{G}_{\mathrm{WP}}\equiv_{W}\mathrm{Lim}'.
\]
\end{thm}

The proofs are via direct reductions and do not rely on previous lemmas. 

We use the construction used to prove Theorem \ref{thm:Finite nilpotent FP}
to also answer a problem from \cite{Benli2019}: 
\begin{prop}
The set $\mathcal{R}\mathrm{Fin}$ of residually finite groups is
$\Pi_{3}^{0}$ complete in the Borel hierarchy on the space of marked
groups. So are the set of residually nilpotent and residually finitely
presented groups. 
\end{prop}

\subsection{Residually finite image }

We use a construction of Dyson from \cite{Dyson1974} to classify
$\mathrm{Res}_{\mathrm{Fin}}$, where $\mathrm{Fin}$ is the set of
finite groups. 

\subsubsection{Dyson's doubles of the lamplighter group}

The lamplighter group $L$ is the wreath product of $\mathbb{Z}/2$
and $\mathbb{Z}$, noted $\mathbb{Z}/2\wr\mathbb{Z}$, which is by
definition the semi-direct product $\underset{\mathbb{Z}}{\bigoplus}\mathbb{Z}/2\mathbb{Z}\rtimes\mathbb{Z}$,
where $\mathbb{Z}$ acts on $\underset{\mathbb{Z}}{\bigoplus}\mathbb{Z}/2\mathbb{Z}$
by permuting the indices. It admits the following presentation:
\[
\langle a,t\vert\,a^{2},\,\left[a,t^{i}at^{-i}\right],i\in\mathbb{Z}\rangle
\]
The element $u_{i}\overset{\text{def}}{=}t^{i}at^{-i}$ of $L$ corresponds
to the element of $\underset{\mathbb{Z}}{\bigoplus}\mathbb{Z}/2\mathbb{Z}$
with only one non-zero coordinate in position $i\in\mathbb{Z}$. Consider
another copy $\hat{L}$ of the lamplighter group, together with an
isomorphism from $L$ to $\hat{L}$ we note $g\mapsto\hat{g}$. For
each subset $\mathcal{A}$ of $\mathbb{Z}$, define $L(\mathcal{A})$
to be the amalgamated product of $L$ and $\hat{L}$, with $u_{i}=t^{i}at^{-i}$
identified with $\hat{u}_{i}=\hat{t}^{i}\hat{a}\hat{t}^{-i}$ for
each $i$ in $\mathcal{A}$. It has the following presentation:
\[
\langle a,\hat{a},t,\hat{t}\vert\,a^{2},\,\hat{a}^{2},\,\left[a,t^{i}at^{-i}\right],\left[\hat{a},\hat{t}^{i}\hat{a}\hat{t}^{-i}\right],i\in\mathbb{Z},\,t^{j}at^{-j}=\hat{t}{}^{j}\hat{a}\hat{t}{}^{-j},\,j\in\mathcal{A}\rangle.
\]

Recall that the profinite topology on $\mathbb{Z}$ is defined as
follows: $U$ is open if for any $z\in U$ there is $k\in\mathbb{Z}$
such that $z+k\mathbb{Z}\subseteq U$. For $\mathcal{A}\subseteq\mathbb{Z}$,
we denote by $\overline{\mathcal{A}}$ its closure in the profinite
topology. 
\begin{prop}
[Dyson, \cite{Dyson1974}, Theorem 1]For any $\mathcal{A}\subseteq\mathbb{Z}$,
we have 
\[
\mathrm{Res}_{\mathrm{Fin}}(L(\mathcal{A}))=L(\overline{\mathcal{A}}).
\]
In particular, $L(\mathcal{A})$ is residually finite if and only
if $\mathcal{A}$ is closed in the profinite topology on $\mathbb{Z}$.
\end{prop}

The following result is immediate: 
\begin{prop}
\label{prop:L is an homeo }The map
\begin{align*}
L:\mathcal{P}(\mathbb{Z}) & \rightarrow\mathcal{G}\\
A & \mapsto L(A)
\end{align*}
is injective, and it defines a homeomorphism onto its image when $\mathcal{P}(\mathbb{Z})$
and $\mathcal{G}$ have either their Scott topologies ($\mathbb{S}^{\mathbb{Z}}$
and $\mathcal{G}_{\mathrm{Pres}}$), or their Polish topologies ($\{0,1\}^{\mathbb{Z}}$
and $\mathcal{G}_{\mathrm{WP}}$). \qed
\end{prop}

As a corollary of the above, we have a reduction: $\mathrm{Res}_{\mathrm{Fin}}$
reduces to the closure map $C:\mathcal{P}(\mathbb{Z})\rightarrow\mathcal{P}(\mathbb{Z}),A\mapsto\overline{A}$. 
\begin{cor}
\label{cor: reduc via Dyson's groups}There is a Weihrauch reduction
\[
C\le_{W}\mathrm{Res}_{\mathrm{Fin}}.
\]
\end{cor}

\begin{proof}
By Proposition \ref{prop:L is an homeo }, we can introduce $L^{-1}$,
which is a continuous map from the space of marked groups to $\mathcal{P}(\mathbb{Z})$
with the Cantor space topology. The reduction is then given by: $C(\mathcal{A})=L^{-1}(\mathrm{Res}_{\mathrm{Fin}}(L(\mathcal{A})))$
for any $\mathcal{A}\subseteq\mathbb{Z}$.
\end{proof}

\subsubsection{Classification of $\mathrm{Res}_{\mathrm{Fin}}$}
\begin{lem}
\label{lem:classifying closure in ptz }We have the Weihrauch equivalence:
$C\equiv_{W}\mathrm{Lim}'$ .
\end{lem}

\begin{proof}
We first remark that $\mathbb{Z}$ equipped with the profinite topology
is homeomorphic to $\mathbb{Q}$ equipped with the Euclidian topology:
by classical results of Sierpirlski and Frechet, $\mathbb{Q}$ is,
up to homeomorphism, the only nonempty, countable, metrizable and
perfect space (see \cite[7.12]{Kechris1995}). Thus $C$ is equivalent
to the closure map on $\mathbb{Q}$, which we denote by $C_{\mathbb{Q}}$:
\begin{align*}
C_{\mathbb{Q}}:\mathcal{P}(\mathbb{Q})\rightarrow\mathcal{P}(\mathbb{Q}),\,A\mapsto\overline{A}.
\end{align*}
(The representation of $\mathbb{Q}$ that we use is associated to
the discrete topology, where $\pm\frac{p}{q}$ is given by the triple
$(\pm1,p,q)$.)

The problem $\mathrm{Lim}'$ is equivalent to the infinite parallelization
of the problem $\mathrm{LPO}'$ \cite{Brattka2021a}. In other words,
the input is a double sequence $(u_{n,i})_{(n,i)\in\mathbb{N}^{2}}$
of digits in $\{0,1\}$, and its output is the sequence $(v_{i})_{i\in\mathbb{N}}\in\{0,1\}^{\mathbb{N}}$
given by 
\[
v_{i}=1\iff\exists^{\infty}n,\,u_{n,i}=1.
\]

We first show that $C_{\mathbb{Q}}\ge_{W}\mathrm{Lim}'$. 

For a double sequence $(u_{n,i})_{(n,i)\in\mathbb{N}^{2}}$, define
a subset $A$ of $\mathbb{Q}$ by 
\[
i+\frac{1}{n+2}\in A\iff u_{n,i}=1.
\]
If follows immediately that $i\in\overline{A}\iff\exists^{\infty}n,\,u_{n,i}=1$.
This provides the desired reduction. 

Conversely, we show that $\mathrm{Lim}'\ge_{W}C_{\mathbb{Q}}$.

Consider $A\subseteq\mathbb{Q}$ and $r\in\mathbb{Q}$. We then have
\[
r\in\overline{A}\iff(\forall n\in\mathbb{N}\exists q\in A,\,\vert q-r\vert<2^{-n}).
\]
This $\forall\exists$ statement can be converted to an instance of
$\mathrm{LPO}'$: we define a double sequence $d_{n,p}$ given by
$d_{n,p}=1$ if there exists a rational $q$ with numerator and denominator
both at most $p$ which is in $A$ and at distance at most $2^{-n}$
of $r$, and if furthermore $d_{n,t}=0$ for $t<p$. We then have
\[
r\in\overline{A}\iff\exists^{\infty}(n,p),d_{n,p}=1
\]
It follows that $C_{\mathbb{Q}}$ is computed by the parallelization
of $\mathrm{LPO}'$, which is precisely $\mathrm{Lim}'$. 
\end{proof}
\begin{cor}
\label{cor:Res Fin =00003D Lim'}We have the equivalence
\[
\mathrm{Lim}'\equiv_{W}\mathrm{Res}_{\mathrm{Fin}}.
\]
\end{cor}

\begin{proof}
The reduction $\mathrm{Res}_{\mathrm{Fin}}\le_{W}\mathrm{Lim}'$ follows
from Proposition \ref{prop:Countable}. 

The fact that $\mathrm{Lim}'\le_{W}\mathrm{Res}_{\mathrm{Fin}}$ follows
from $C\equiv_{W}\mathrm{Lim}'$ (Lemma \ref{lem:classifying closure in ptz })
and $C\le_{W}\mathrm{Res}_{\mathrm{Fin}}$ (Corollary \ref{cor: reduc via Dyson's groups}). 
\end{proof}

\subsubsection{Relation with the residual finiteness growth }

Let $(G,S)$ be a residually finite marked group. The \emph{residual
finiteness growth} \emph{of} $(G,S)$ \cite{BouRabee2010} is the
function which maps $n$ to the smallest $k$ such that every non-trivial
element of length at most $n$ in $(G,S)$ has a non trivial image
in a finite quotient of $G$ of size at most $k$. 

Groups with arbitrarily large residual finiteness growth were constructed
in \cite{BouRabee2016}. Here we note that the proof of the fact that
$\mathrm{Lim}'\equiv_{W}\mathrm{Res}_{\mathrm{Fin}}$ necessarily
relies on arbitrarily rapid residual finiteness growth. 

Note that the construction of Neumann used in \cite{BouRabee2016}
produces only residually finite groups. Thus this construction cannot
be used to prove $\mathrm{Lim}'\equiv_{W}\mathrm{Res}_{\mathrm{Fin}}$.
Another group construction where arbitrary residual finiteness growth
is obtained, but which does not produce only residually finite groups,
is given in \cite{Darbinyan2025}, this construction could probably
be used to establish the equivalence $\mathrm{Lim}'\equiv_{W}\mathrm{Res}_{\mathrm{Fin}}$. 
\begin{prop}
Consider the restriction $(\mathrm{Res}_{\mathrm{Fin}})_{\vert\mathcal{A}}$
of the problem $\mathrm{Res}_{\mathrm{Fin}}$ to a set $\mathcal{A}$
of groups, such that there is a universal asymptotic upper bound $f$
to the residual finiteness growth of the groups in $\mathrm{Res}_{\mathrm{Fin}}(\mathcal{A})$. 

Then $(\mathrm{Res}_{\mathrm{Fin}})_{\vert\mathcal{A}}\le_{W}\mathrm{Lim}_{\mathbb{N}}\ast\mathrm{Lim}=\mathrm{Lim}_{\mathbb{N}}'$,
and thus $(\mathrm{Res}_{\mathrm{Fin}})_{\vert\mathcal{A}}<_{W}\mathrm{Lim}'$. 
\end{prop}

\begin{proof}
Let $(G,S)$ be a marked group in $\mathcal{A}$. 

For each $n$ we try to build the ball of radius $n$ of the Cayley
graph of $\mathrm{Res}_{\mathrm{Fin}}((G,S))$. 

We act as though $f$ was an actual upper bound to the residual finiteness
growth of $(G,S)$ (and not only an asymptotic upper bound). 

List all finite marked groups of size at most $f(n)$. By applying
LPO to each of them, it is possible to determine exactly those that
are marked quotients of $(G,S)$. We then build a finite graph $\Gamma_{n}$
by saying that $w$ is non-trivial in $\Gamma_{n}$ if and only if
it has a non trivial in one of the listed marked finite groups. 

By doing the above in parallel for each $n$, we get a sequence $\Gamma_{0}$,
$\Gamma_{1}$, $\Gamma_{2}$... of finite graphs. This sequence can
be obtained by solving in parallel infinitely many instances of LPO,
i.e., by applying Lim. Because $f$ is an asymptotic upper bound to
the residual finiteness growth of $(G,S)$, there is some $N$ such
that each $\Gamma_{k}$, $k\ge N$, is actually the ball of radius
$k$ in the Cayley graph of $(G,S)$. 

Note that this $N$ is reached exactly when contradictions stop arising
between the graphs $\Gamma_{i}$: 
\[
N=\inf\{k\mid\forall n\ge k,\,\Gamma_{n+1}\text{ properly extends }\Gamma_{n}\}.
\]
Define a sequence $(u_{n})_{n\in\mathbb{N}}$ by $u_{n+1}=u_{n}$
if $\Gamma_{n+1}$ properly extends $\Gamma_{n}$, and $u_{n+1}=u_{n}+1$
otherwise. This sequence depends continuously on the sequence $(\Gamma_{n})_{n\in\mathbb{N}}$,
and it converges. One application of $\mathrm{Lim}_{\mathbb{N}}$
can provide the limit of this sequence. And the sequence $(\Gamma_{n})_{n\in\mathbb{N}}$
together with this limit provide exactly the Cayley graph of $\mathrm{Res}_{\mathrm{Fin}}((G,S))$.
\end{proof}

\subsubsection{Two more uses of Dyson's construction }

Here we prove two additional results based on Dyson's construction:
there exists a group INIP relative to the quasivariety of LEF groups
(a fact used in Corollary \ref{cor: Res LEF k-solvable torsion free }),
and that the set of residually finite groups is $\boldsymbol{\Pi}_{3}^{0}$-complete
in the space of marked groups (a question left open in \cite{Benli2019},
where the upper bound $\boldsymbol{\Pi}_{4}^{0}$ was given). 
\begin{prop}
For every $A\subseteq\mathbb{Z}$, the group $L(A)$ is LEF. 
\end{prop}

\begin{proof}
If $A$ is closed in the profinite topology on $\mathbb{Z}$, then
$L(A)$ is even residually finite. If $A$ is not closed, $A$ is
the limit, in the Cantor space topology on $\mathcal{P}(\mathbb{Z})$,
of a sequence of closed sets: take for instance finite subsets of
$A$ that exhaust it. Because $L$ is continuous, $L(A)$ is the limit
in $\mathcal{G}_{\mathrm{WP}}$ of a sequence of residually finite
groups, thus it is LEF. 
\end{proof}
\begin{prop}
\label{prop:INIP for LEF groups}The group $L(\emptyset)$ is INIP
relative to the set of LEF groups. 
\end{prop}

\begin{proof}
Consider an infinite subset $A$ of $\mathbb{Z}$ such that for every
$B\subseteq A$, $B$ is closed in the profinite topology of $\mathbb{Z}$.
For instance $A$ could be the set of primes. In this case, the relations
$u_{i}=\hat{u}_{i}$, $i\in A$, are independent modulo the set of
LEF groups. Indeed, for each $i\in A$, the group $L(A\setminus\{i\})$
is LEF, and $u_{i}=\hat{u}_{i}$ is not a relation in it. 
\end{proof}
\begin{prop}
\label{prop:Res fin Pi_3}The set of residually finite groups is $\boldsymbol{\Pi}_{3}^{0}$-complete
in the space of marked groups. 
\end{prop}

\begin{proof}
Being residually finite is in $\boldsymbol{\Pi}_{3}^{0}$. A marked
group $(G,S)$ is residually finite if for any non-trivial element
$w\in G\setminus\{1\}$ there is some finite marked group $(F,S')$
in which $w\ne1$ such that every relation of $(G,S)$ is a relation
of $(F,S')$:
\begin{align*}
\forall w\in\mathbb{F}_{S}, & \exists(F,S')\,\text{finite},\,\forall r\in\mathbb{F}_{S},\\
 & w\neq_{(G,S)}1\implies(w\neq_{(F,S')}1\,\&\,(r=_{(G,S)}1\implies r=_{(F,S')}1)).
\end{align*}

We now prove completeness. 

Via Dyson's construction, we get that $L(A)$ is residually finite
if and only if $A$ is closed in the profinite topology on $\mathbb{Z}$.
Recall that $\mathbb{Z}$ equipped with the profinite topology is
homeomorphic to $\mathbb{Q}$ with the subset topology of $\mathbb{R}$
(proof of Lemma \ref{lem:classifying closure in ptz }). We thus get
a Wadge reduction between the set of residually finite groups and
the set of closed subsets of $\mathbb{Q}$ (in the Polish space of
subsets of $\mathbb{Q}$ equipped with the prodiscrete topology). 

And thus it suffices to show that being closed in the Euclidian topology
of $\mathbb{Q}$ is a $\boldsymbol{\Pi}_{3}^{0}$ complete property. 

We Wadge reduce the $\boldsymbol{\Pi}_{3}^{0}$ -complete set 
\[
S_{3}=\{x\in2^{\mathbb{N}\times\mathbb{N}}\mid\forall n,\forall^{\infty}m,\,x(n,m)=0\}
\]

to the set of closed subsets of $\mathbb{Q}$. For $x\in2^{\mathbb{N}\times\mathbb{N}}$,
we define a set of rationals $A_{x}$ via 
\[
n+\frac{1}{m+1}\in A_{x}\iff x(n,m)=1.
\]
It is easy to see that the map $x\mapsto A_{x}$ is continuous, and
that $A_{x}$ is closed if and only if $x\in S_{3}$. 
\end{proof}

\subsection{\label{subsec:Hall's-central-by-metabelian-gro}Hall's central-by-metabelian
group }

The following construction due to P. Hall \cite{Hall1954} will be
useful for both the classification of the maps associated to residually
nilpotent and residually finitely presented images. 

We first describe it. In what follows, we use the convention that
$[x,y]=x^{-1}y^{-1}xy$, and define inductively $[x,y,z]=[[x,y],z]$,
$[x,y,z,v]=[[x,y,z],v]$ and so on. Furthermore we will write $x^{y}$
for $y^{-1}xy$.

Hall's group $G$ is generated by two elements $a$ and $b$. For
$i$ in $\mathbb{Z}$, denote by $b_{i}$ the element $a^{-i}ba^{i}$.
The group $G$ is then given by the presentation 
\[
\pi_{1}=\langle a,b\,\vert\,[b_{i},b_{j},a]=1,~\,[b_{i},b_{j},b]=1,~\,\forall i,j\in\mathbb{Z}\rangle.
\]
Note that $[b_{i},b_{j}]^{a}=[b_{i+1},b_{j+1}]$, so 
\[
[b_{i},b_{j},a]=1\iff[b_{i},b_{j}]=[b_{i+1},b_{j+1}],
\]
and thus the above presentation could be written equivalently as 
\[
\pi_{1}'=\langle a,b\,\vert\,[b,b_{i},a]=1,~\,[b,b_{i},b]=1,~\,\forall i\in\mathbb{Z}\rangle.
\]
The group $G$ can both be seen as a maximal central extension of
the wreath product $\mathbb{Z}\wr\mathbb{Z}$ and as an HNN extension
of an infinitely generated two step nilpotent group generated by the
elements $b_{i}$, $i\in\mathbb{Z}$. The group $G$ is thus center-by-metabelian
and (class 2 nilpotent)-by-abelian. 

Denote by $d_{i}$ the element $[b_{0},b_{i}]$, for $i\in\mathbb{Z}$.
Then $d_{i}=d_{-i}^{-1}$ and $d_{i}=[b_{t},b_{t+i}]$ for all $t\in\mathbb{Z}$.

The center of $G$ is a free abelian group on countably many generators
\cite{Hall1954}, in fact the elements $d_{i}$, $i>0$, form the
basis of the center of $G$.

In \cite{Bartholdi2025}, it was shown that, to consider nilpotent
quotients of $G$, using another basis for the center is useful. Consider
the element $f_{n}$ given by 
\[
f_{n}=[b,\underbrace{a,...,a}_{n},b].
\]
Denote by $\gamma_{n}(G)$ the $n$-th term in the lower central series
of $G$, given by $\gamma_{1}(G)=G$ and $\gamma_{n+1}(G)=[\gamma_{n}(g),G]$.
A group $H$ is residually nilpotent if and only if 
\[
\bigcap_{n\ge1}\gamma_{n}(H)=\{1\}.
\]

\begin{lem}
[{\cite[Lemma 68, Lemma 70]{Bartholdi2025}}]In the notations introduced
above: 
\begin{enumerate}
\item The elements $f_{n}$ form another basis of the center of $G$; 
\item Each $f_{n}\in\gamma_{n+2}(G)$;
\item And $G$ is residually nilpotent, i.e., there is some function\footnote{In fact, for $n$ odd $f_{n}\in\gamma_{n+2}(G)\setminus\gamma_{n+3}(G)$,
but for $n$ even no precise bound for $h$ was computed in \cite{Bartholdi2025}.} $h$ s.t. $f_{n}\in\gamma_{n+2}(G)\setminus\gamma_{h(n)}(G)$. 
\end{enumerate}
\end{lem}

It follows that a central quotient of $G$ obtained by adding relations
of the form $f_{i}=f_{j}$ is residually nilpotent if and only if,
for each $i$, $f_{i}$ is identified to only finitely many $f_{j}$,
$j\ne i$. 

\subsection{\label{subsec:Residually-nilpotent-image}Residually nilpotent image }

By Proposition \ref{prop:Countable}, $\mathrm{Res}_{\mathcal{N}}\le_{W}\widehat{\mathrm{LPO}}'$.
\begin{lem}
Let $\mathcal{N}$ be the set of finitely generated nilpotent groups.
Then $\mathrm{Res}_{\mathcal{N}}\ge_{W}\widehat{\mathrm{LPO}}'$. 
\end{lem}

\begin{proof}
A double sequence $u=(u_{n,m})\in\{0,1\}^{\mathbb{N}^{2}}$ is mapped,
via $\widehat{\mathrm{LPO}}'$, to the sequence $(v_{n})\in\{0,1\}^{\mathbb{N}}$
given by 
\[
v_{n}=1\iff\exists^{\infty}m,\,u_{n,m}=1.
\]

Associated to the double sequence $u$, we consider a central quotient
of Hall's group $G$, obtained by adding to it the following relations:
\[
f_{\langle n,0\rangle}=f_{\langle n,m\rangle}\text{ for each \ensuremath{m} s.t. }u_{n,m}=1.
\]
(Recall that $n,m\mapsto\langle n,m\rangle$ is Cantor's pairing function.)
Denote by $G_{u}$ the obtained group. It is easy to see that the
word-problem of $G_{u}$ can be continuously constructed from the
sequence $u$. And:
\begin{itemize}
\item If there are only finitely many $m$ with $u_{n,m}=1$, $f_{\langle n,0\rangle}$
has a non-trivial image in a nilpotent quotient of $G_{u}$.
\item If there are infinitely many $m$ with $u_{n,m}=1$, $f_{\langle n,0\rangle}\in\gamma_{m}(G)$
for every $m$, and thus $f_{\langle n,0\rangle}$ is trivial in the
residually nilpotent image of $G$. 
\end{itemize}
It follows from this that $f_{\langle n,0\rangle}=1$ in $\mathrm{Res}_{\mathcal{N}}((G_{u},(a,b)))$
if and only if $\exists^{\infty}m,\,u_{n,m}=1$. This gives the desired
reduction. 
\end{proof}
\begin{prop}
The set of residually nilpotent groups is $\boldsymbol{\Pi}_{3}^{0}$-complete
in the space of marked groups. 
\end{prop}

\begin{proof}
The proof is identical to that of Proposition \ref{prop:Res fin Pi_3},
replacing the use of Dyson's doubles of the lamplighter group by central
quotients of Hall's group, as above. 
\end{proof}

\subsection{Residually finitely presented image }

Because of Higman's Embedding Theorem, we rely on notions coming from
computability theory to classify $\mathrm{Res}_{\mathcal{FP}}$, where
$\mathcal{FP}$ is the set of finitely presented groups.
\begin{lem}
Let $\mathcal{FP}$ be the set of finitely presented groups. Then
$\mathrm{Res}_{\mathcal{FP}}\ge\widehat{\mathrm{LPO}}'$. 
\end{lem}

\begin{proof}
Consider, as input of $\widehat{\mathrm{LPO}}'$, a sequence $u=(u_{n,m})\in\{0,1\}^{\mathbb{N}^{2}}$,
we must construct a sequence $(v_{n})\in\{0,1\}^{\mathbb{N}}$ with
\[
v_{n}=1\iff\exists^{\infty}m,\,u_{n,m}=1.
\]

We again consider a central quotient of Hall's group $G$, using the
basis $(f_{n})_{n\in\mathbb{N}}$ of the center which is appropriate
to study nilpotent quotients (see Section \ref{subsec:Hall's-central-by-metabelian-gro}). 

Let $A\subseteq\mathbb{N}$ be a simple set, i.e. a c.e. set whose
complement, while infinite, does not contain an infinite c.e. set. 

We suppose $0\notin A$. Let $a_{0}\le a_{1}\le a_{2}...$ be the
(non-computable) enumeration of the complement of $A$. 

For every $n\in\mathbb{N}$ and $m\in A$, we add the relation $f_{\langle n,m\rangle}=1$
to $G$. 

For every $n\in\mathbb{N}$ and $m\in\mathbb{N}$, if $u_{n,m}=1$,
we add the relation $f_{\langle n,0\rangle}=f_{\langle n,a_{m}\rangle}$.

Denote by $G_{u}$ the obtained group. It is easy to see that the
word-problem of $G_{u}$ can be continuously constructed from the
sequence $u$. 

And: 
\begin{itemize}
\item If there are only finitely many $m$ with $u_{n,m}=1$, $f_{\langle n,0\rangle}$
has a non-identity image in a nilpotent quotient of $G_{u}$ (see
Section \ref{subsec:Residually-nilpotent-image}), this quotient is
finitely presentable.
\item If there are infinitely many $m$ with $u_{n,m}=1$, an infinite number
of relations of the form $f_{\langle n,0\rangle}=f_{\langle n,a_{m}\rangle}$
were added. In any recursively presented quotient of $G_{u}$ (and
in particular any finitely presented quotient), the set $\{p\mid f_{\langle n,0\rangle}=f_{\langle n,p\rangle}\}$
is a c.e. set. Thus it cannot be contained in the complement of $A$,
and it must meet $A$, and thus $f_{\langle n,0\rangle}=1$ in this
group. 
\end{itemize}
It follows from this that $f_{\langle n,0\rangle}=1$ in the residually
finitely presented image of $G_{u}$ if and only if $\exists^{\infty}m,\,u_{n,m}=1$.
This gives the desired reduction. 
\end{proof}
\begin{prop}
The set of residually finitely presented groups is $\boldsymbol{\Pi}_{3}^{0}$-complete
in the space of marked groups. 
\end{prop}

\begin{proof}
The proof is identical to that of Proposition \ref{prop:Res fin Pi_3},
replacing the use of Dyson's doubles of the lamplighter group by central
quotients of Hall's group, as above. 
\end{proof}

\section{Above $\mathrm{Lim}'_{\mathbb{N}^{\mathbb{N}}}$}

Proposition \ref{prop:Arbitrarily-complicated Res C} provides examples
of sets $\mathcal{C}$ such that $\mathrm{Res}_{\mathcal{C}}$ is
not measurable, and thus in particular not not $\Pi_{3}^{0}$-measurable. 

It would be interesting to classify precisely some problems $\mathrm{Res}_{\mathcal{C}}$
that are strictly above $\mathrm{Lim}'_{\mathbb{N}^{\mathbb{N}}}$.
We leave the following as a problem. 
\begin{problem}
Prove that, if $\mathcal{C}$ is one of: the set of finitely generated
solvable groups, the set of finitely generated torsion groups, then
$\mathrm{Res}_{\mathcal{C}}$ is not $\Pi_{3}^{0}$-measurable.
\end{problem}

\bibliographystyle{alpha}
\bibliography{TheOneBib}

\end{document}